\documentclass[preprint]{siamart190516}

\usepackage{amsfonts,empheq}
\usepackage{nomencl} 
\makenomenclature	
\usepackage[toc,page]{appendix} 
\usepackage{graphicx} 
\usepackage{longtable}
\usepackage{setspace}
\usepackage{framed}
\usepackage{cite}
\usepackage{fancyhdr}
\usepackage{pdfpages}
\usepackage[a4paper,dvips,left=3.5cm, right=3cm, top=2.5cm, bottom=3.5cm]{geometry}
\usepackage{siunitx} 
\usepackage{paralist} 
\usepackage{multicol}
\usepackage{wrapfig}
\usepackage[nameinlink,capitalize]{cleveref}

\usepackage{tikz}
\usetikzlibrary{backgrounds,snakes,decorations,positioning,arrows,shapes.misc}
\tikzset{
	block/.style={
		draw, 
		rectangle, 
		minimum height=1.5cm, 
		minimum width=2.7cm, align=center
	}, 
	line/.style={->,>=latex'}, cross/.style={cross out, draw=black, minimum size=2*(#1-\pgflinewidth), inner sep=0pt, outer sep=0pt},
	cross/.default={1pt}
}
\tikzset{cross/.style={cross out, draw=black, minimum size=2*(#1-\pgflinewidth), inner sep=0pt, outer sep=0pt},
	cross/.default={1pt},
	position label/.style={
		below = 3pt,
		text height = 1.5ex,
		text depth = 1ex
	},
	brace/.style={
		decoration={brace, mirror},
		decorate
	}
}
\usepackage[font=footnotesize,labelfont=footnotesize]{caption}
\usepackage[font=footnotesize,labelfont=footnotesize]{subcaption}

\usepackage{xstring}
\usepackage{xpatch}
\patchcmd{\thenomenclature}
{\leftmargin\labelwidth}
{\leftmargin\labelwidth\itemindent 0em}
{}{}

\newcommand{\nomenclheader}[1]{%
	\item[\hspace*{-\itemindent}\normalfont\bfseries#1]}
\renewcommand\nomgroup[1]{%
	\IfStrEqCase{#1}{%
		{M}{\nomenclheader{Mathematical symbols}}
		{G}{\nomenclheader{Greek symbols}}
		{R}{\nomenclheader{Roman symbols}}
		{O}{\nomenclheader{Other symbols}}
	}%
}

\newsiamremark{remark}{Remark}
\newsiamremark{prop}{Proposition}
\newcommand{\R}{\mathbb{R}}
\newcommand{\J}{\mathcal{J}}

\newcommand{\W}{\mathcal{W}}
\newcommand{\LBA}{\mathfrak{L}} 
\newcommand{\Co}{\mathcal{C}^0}
\newcommand{\C}{\mathcal{C}}

\newcommand{\Parttwo}[2]{\dfrac{\partial #1}{\partial #2}}
\newcommand{\str}{\sp{\prime}}
\newcommand{\esssup}{{\mathrm{ess}\sup}}
\let\div\relax
\DeclareMathOperator{\div}{div}

\allowdisplaybreaks[1] 
\newcolumntype{C}[1]{>{\centering\arraybackslash}p{#1}}
\DeclareMathOperator{\atan}{atan} 

\usepackage[square,sort, comma,numbers]{natbib} 
\title{Existence, uniqueness and numerical modeling of wine fermentation based on 
	integro-differential equations\thanks{Preprint submitted to SIAM Journal on Applied Mathematics; Parts of this work have been part of the PhD thesis of the first author.}}

\author{C. Schenk\thanks{BCAM - Basque Center for Applied Mathematics, Mazarredo 14, E48009 Bilbao, Basque Country - Spain (\email{cschenk@bcamath.org}, corresponding author).}
\and V.H. Schulz\thanks{Trier University, Department of Mathematics,
	54286 Trier, Germany (\email{volker.schulz@uni-trier.de}).}}

\begin{document}
		\maketitle
	\begin{abstract}Predictive modeling is the key factor for saving time and resources with respect to manufacturing processes such as fermentation processes arising e.g.\ in food and chemical manufacturing processes. According to \citet{zhang2002CellPop} the open-loop dynamics of yeast are highly dependent on the initial cell mass distribution. This can be modeled via population balance models describing the single-cell behavior of the yeast cell. There have already been several population balance models for wine fermentation in the literature. However, the new model introduced in this paper is much more detailed than the ones studied previously. This new model for the white wine fermentation process is based on a combination of components previously introduced in literature. It turns it into a system of highly nonlinear weakly hyperbolic partial/ordinary integro-differential equations. This model becomes very challenging from a theoretical and numerical point of view. Existence and uniqueness of solutions to a simplified version of the introduced problem is studied based on semigroup theory. For its numerical solution a numerical methodology based on a finite volume scheme combined with a time implicit scheme is derived. The impact of the initial cell mass distribution on the solution is studied and underlined with numerical results. The detailed model is compared to a simpler model based on ordinary differential equations. The observed differences for different initial distributions and the different models turn out to be smaller than expected. The outcomes of this paper are very interesting and useful for applied mathematicians, winemakers and process engineers.\end{abstract}
		\section{Introduction}
	Many processes arising in diverse contexts can be classified as fermentation process\-es, such as manufacturing of some food products and of some industrial and pharmaceutical chemicals. An overview of fermentation products can be found in \citet{Chojnacka2011FermentationP}. This work deals with the process of white wine fermentation.\\
	In the literature many models based on ordinary differential equations (ODEs) exist to model the process of wine fermentation e.g. \citet{david10, david11,velten2015modellsimulation}. \citet{MILLER2020109783} recently published a review looking at several different ways of modeling the process of wine fermentation also including heat transfer and convective mixing but mostly focused on red wine fermentation.
	However, \citet{zhang2002CellPop} claimed that the open-loop dynamics of yeast are highly dependent on the initial cell mass distribution. This is why, we introduce a new model consisting of a combination of different components from the literature. This model describes the wine fermentation process taking the yeast cell growth dynamics into account. It is based on partial/ordinary integro-differential equations.\\ 
	Models based on integro-differential equations (IDEs) describing the development of the yeast population taking the single cell into account had already been proposed e.g. by \citet{dynamichens, dynamicdaouthens, dynamicmantzaris}.
	This research work arose from the collaborative project R\OE NOBIO (2013 -- 2017, Robust Energy-Optimization of Fermentation Processes for the Production of Biogas and Wine), where a new model modeling the reaction kinetics had already been proposed \citep{NovMod14, schenk17}. This model is based on Michaelis--Menten kinetics \citep{MichaelisMenten1913} and ODEs. Compared to previous models in the literature, this model takes oxygen and ethanol-related death of yeast into account.\\ 
	The IDE model investigated in this study combines these two modeling approaches into one. It is based on the earlier population balance models and describes the rates related to the reaction kinetics and death of yeast cells as in the new ODE model \citep{schenk17}. Thereby, a partial IDE describes the population balance and ordinary IDEs describe the evolution of the other substrate concentrations and the product concentration.\\
	This new model describing the yeast cell dynamics based on IDEs proposes a challenge from the theoretical and numerical side. That is why, we investigate it in this sense.\\
	The study of existence and uniqueness of a solution of this hyperbolic system is strongly dependent on the characteristics of the advection term and the reaction rates. Here a simplified semilinear hyperbolic problem is studied. For semilinear hyperbolic systems, many different approaches exist, e.g.\ \citet{pazy1992semigroups, engel1999one,qamar2008crystalana,wloka1982partielle}. This work shows that this particular problem admits a unique solution based on the concept of semigroups in a similar fashion as described in \citet{dautray6}. \\
	In the last decades, mass-structured cell population balance models have been solved numerically using various solution approaches ranging from the method of characteristics to finite difference to Galerkin to finite volume schemes. A finite volume method (FVM) was first used for the solution of a nonlinear aggregation-breakage population balance equation by \citet{Kumar2014pbeaggbreak}.
	However, this paper goes one step further as it introduces a finite volume scheme for a nonlinear hyperbolic system based on a population balance equation and several ordinary integro-differential equations, using an upwind scheme for the discretization of the nonlinear advection term. According to the initial motivation to model the single-cell growth dynamics, different initial distributions 
	are studied and compared. The simulation results for this model based on IDEs are compared with the much simpler model based on ODEs \citep{schenk17, schenk2018docthesis}.
	More details of the current study can be found in \citet{schenk2018docthesis}.\\
	In Section~\ref{sec:matmeth} the main model describing the white wine fermentation process is introduced. Furthermore, the numerical methods for the simulation of this system are introduced. Moreover, the outcomes of theoretical investigations regarding existence and uniqueness of the solution to a simplified problem are presented. In Section~\ref{sect:resultsdiscuss} the introduced methods are applied to the white wine fermentation model. The numerical results are discussed for different initial yeast distributions. Conclusions are presented in Section~\ref{sect:concl}. 
	\section{Model, Methods and Theory}\label{sec:matmeth}
	\subsection{Model Derivation}\label{sec:IDEmod}
	The central component of this paper is the model introduced in this section. This model was derived based on the population balance models, introduced in \citet{dynamicdaouthens, dynamichens, dynamicmantzaris, kremling} and the ODE model introduced in \citet{schenk17}. This model was introduced for the first time in \citet{schenk2018docthesis} and represents the yeast cell growth dynamics during wine fermentation taking the single cell into account. The deterministic concept is based on Figure~\ref{fig:cellcycle}.
	\begin{figure}[htbp]
		\centering
		\includegraphics[scale=0.4]{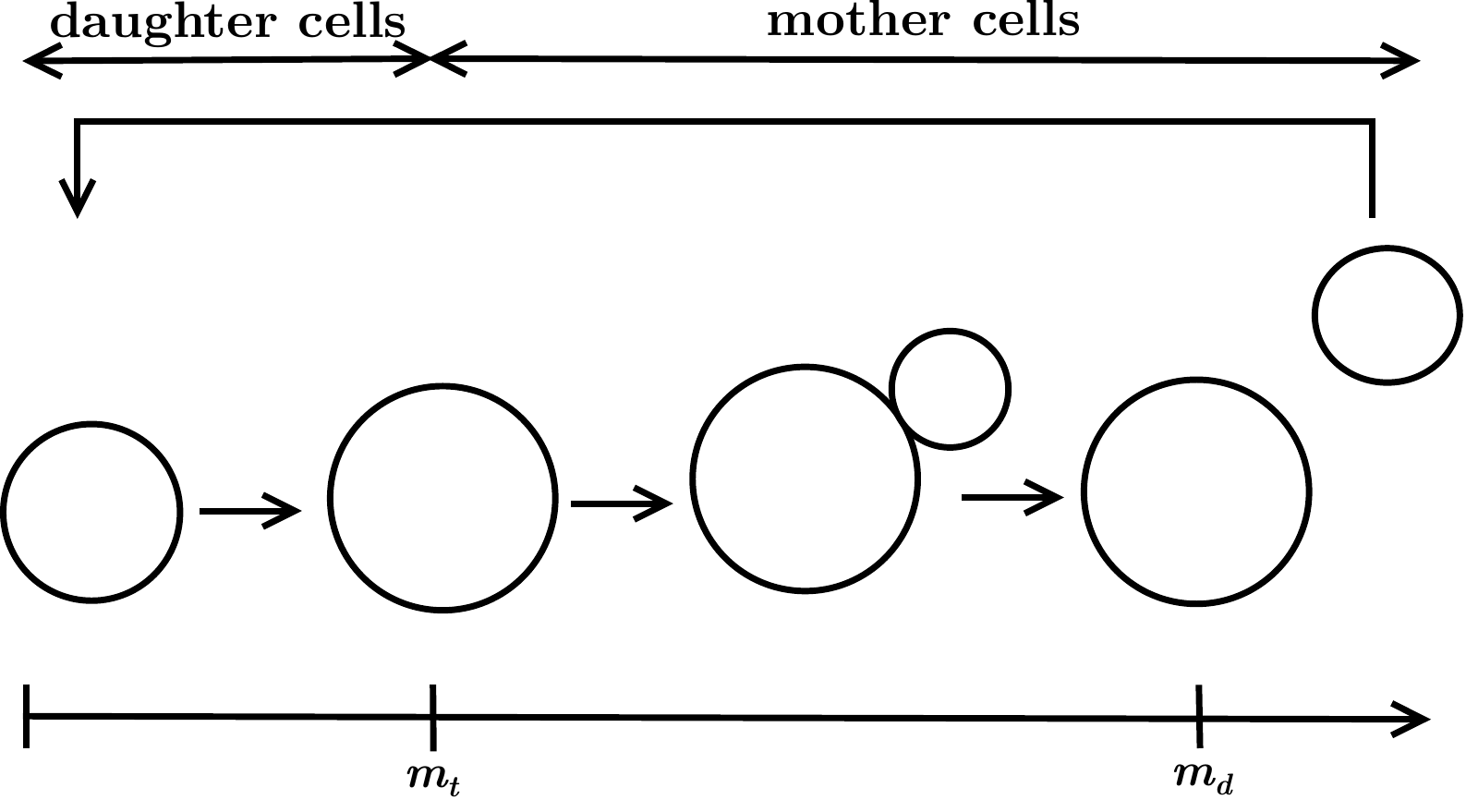}
		\caption{Simplified cell cycle for budding yeast.}
		\label{fig:cellcycle}
	\end{figure}This figure shows the simplified cell cycle for budding yeast. A cell starts as a daughter cell and grows until it is of mass $m_t$ (transient mass) where it becomes a mother cell. It starts budding and when it reaches a mass of $m_d$ (division mass) cell division takes place and it divides into a daughter and mother cell again.
	According to \citet[Chapter~2.2]{morgan2007cell} a yeast cell cycle like this takes about 90 to 120 minutes.\\ 
	The population balance equation (PBE), which models the development of the cell number density depending on its cell mass, is expressed by
\begin{equation}\label{eq:PBE}\begin{split}
	\dfrac{\partial W(m,t)}{\partial t}=&-\frac{\partial(r_{\epsilon}(m,N,S,O)W(m,t))}{\partial m}+2\int_{m_{min}}^{m_{max}}p(m,m\str)\Gamma(m\str)W(m\str,t)dm\str\\&-\Gamma(m)W(m,t)-\Phi(E)W(m,t)-k_d W(m,t)
	\end{split}
\end{equation}
	with the following initial condition 
	\begin{equation}
	W(m,0)=W_0(m)
	\label{eq:PBEIC}
	\end{equation} 
	and boundary conditions
	\begin{equation}
	\begin{split}
	&r_{\epsilon}(m_{min},N,S,O)W(m_{min},t)=0=r_{\epsilon}(m_{max},N,S,O)W(m_{max},t).
	\end{split}
	\label{eq:PBEBC}
	\end{equation}
	These boundary conditions imply that for cells of minimum or maximum mass, i.e. $m_{min}$ or $m_{max}$ respectively, growth is impossible.
	In this model, $m$ is the cell mass and $W(m,t)$ is the cell number density.
	Furthermore, $S, N, O$ and $E$ are the sugar, nitrogen, oxygen and ethanol concentration. The function $p(m,m^{\prime})$ is the partitioning function which allocates the probability of a mother cell $m^{\prime}$ giving birth to a daughter cell $m$. In addition to this, $\Gamma(m)$ is the division rate or in other words the breakage frequency. The single cell growth rate $r_{\epsilon}(m,N,S,O)$ is represented by
	\begin{equation}
	r_{\epsilon}(m,N,S,O)=\mu_{max}(T)\dfrac{N}{K_N+N}\dfrac{S}{K_{S_1}+S}\left(\dfrac{O}{K_{O}+O}+\epsilon\right)m.
	\end{equation}
	This rate is dependent on the constant $\epsilon>0$, as first introduced in \citet{schenk17}. This constant serves the purpose of guaranteeing that other nutrients can still be consumed by the yeast to stay active even without the presence of oxygen.\\
	For modeling the ethanol-related death in detail, the function $\Phi(E)$ with 
	\begin{equation}
	\Phi(E)=\left(0.5 +\dfrac{1}{\pi}\arctan(k_{d_1}(E-tol))\right)k_{d_2}(E-tol)^2
	\label{eq:phi1}
	\end{equation}
	is considered, where $tol>0$ is the tolerance of the ethanol concentration, e.g.\ $tol=79$ g/l, which was determined by a set of data produced at Geisenheim University. Furthermore, $k_{d_1}>0$ and $k_{d_2}>0$ are parameters associated with the death of yeast cells due to ethanol exceeding the tolerance $tol$. This death function is illustrated in Figure~\ref{Fig:PhiE}.
	\begin{figure}[htbp]
		\centering
		\includegraphics[scale=0.4]{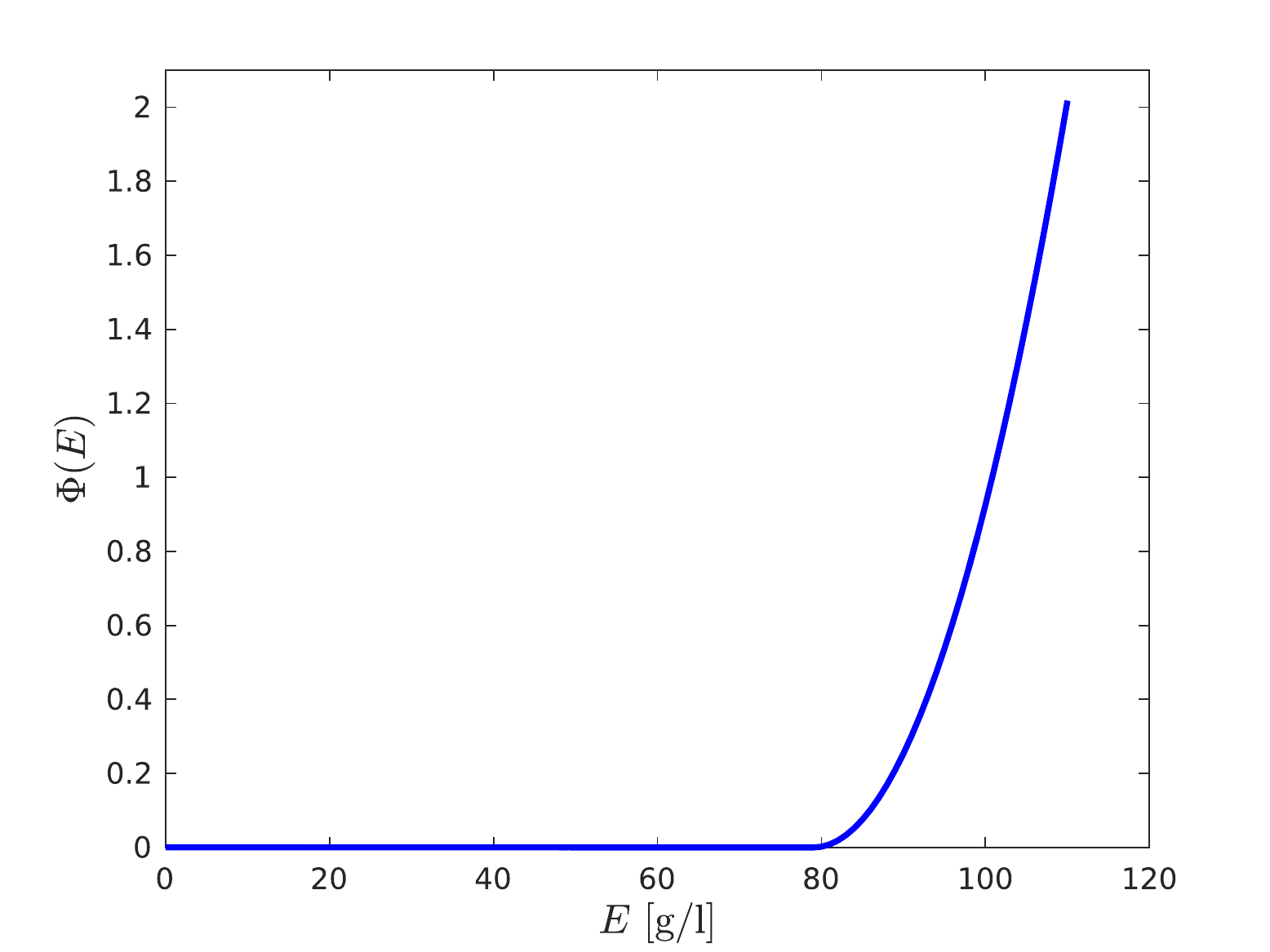}
		\caption{Ethanol-related death function $\Phi(E)$ for $tol=79$ g/l, $k_{d_1}=99.86$, $k_{d_2}=0.0021$ and $E\in[0,110$ g/l$]$.}
		\label{Fig:PhiE}
	\end{figure}
	In detail the partitioning function is represented by
	\begin{equation}
	p(m,m\sp{\prime})=\begin{cases}\begin{aligned}
	&\lambda e^{-\beta(m-m_t)^2}+\lambda e^{-\beta(m-m\sp{\prime}+m_t)^2}&&,\;m\sp{\prime}>m \text{ and } m\sp{\prime}>m_t \\& 0&&,\;\text{else}\end{aligned}\end{cases}
	\end{equation}
	and the division rate is given by
	\begin{equation}
	\begin{split}
	&\Gamma(m)=\begin{cases} \begin{aligned} & 0 &&,\;m\leq m_t\\&\gamma e^{-\delta (m-m_d)^2} &&,\;m_t<m<m_d \\ &\gamma &&,\;\text{else},\end{aligned}\end{cases}\\
	\end{split}
	\end{equation}
	where $m_t$ and $m_d$ are the cell transition and division mass. Moreover, $\lambda$, $\beta$, $\gamma$ and $\delta$ are the deterministic parameters.\\
	The partitioning function $p(m,m\str)$ according to \citet{dynamicmantzaris} should fulfill a normalization condition, i.e.\ 
\begin{equation}\int_{m_{min}}^{m_{max}} p(m,m\str)\;dm=1,\end{equation}
	which assures that it is a density function. Furthermore, biomass should be conserved at cell division according to \citet{dynamicmantzaris}. This means that the following condition should be satisfied for $p(m,m\str)$, namely
 \begin{equation}p(m,m\str)=p(m\str-m,m\str).\end{equation}
	For example the partitioning function and division rate for the values $m_t=0.3784$, $m_d=0.8525$, $\gamma=200$, $\delta=50$, $\lambda=5.6419$, $\beta=400$, fixed $m\sp{\prime}=0.999$ and $m\in[0.001,0.999]$ are illustrated in Figure~\ref{Fig:plotpartgamma}.
	\begin{figure}[htbp]
		\centering
		\vspace{-.5cm}
		\includegraphics[scale=0.5]
		{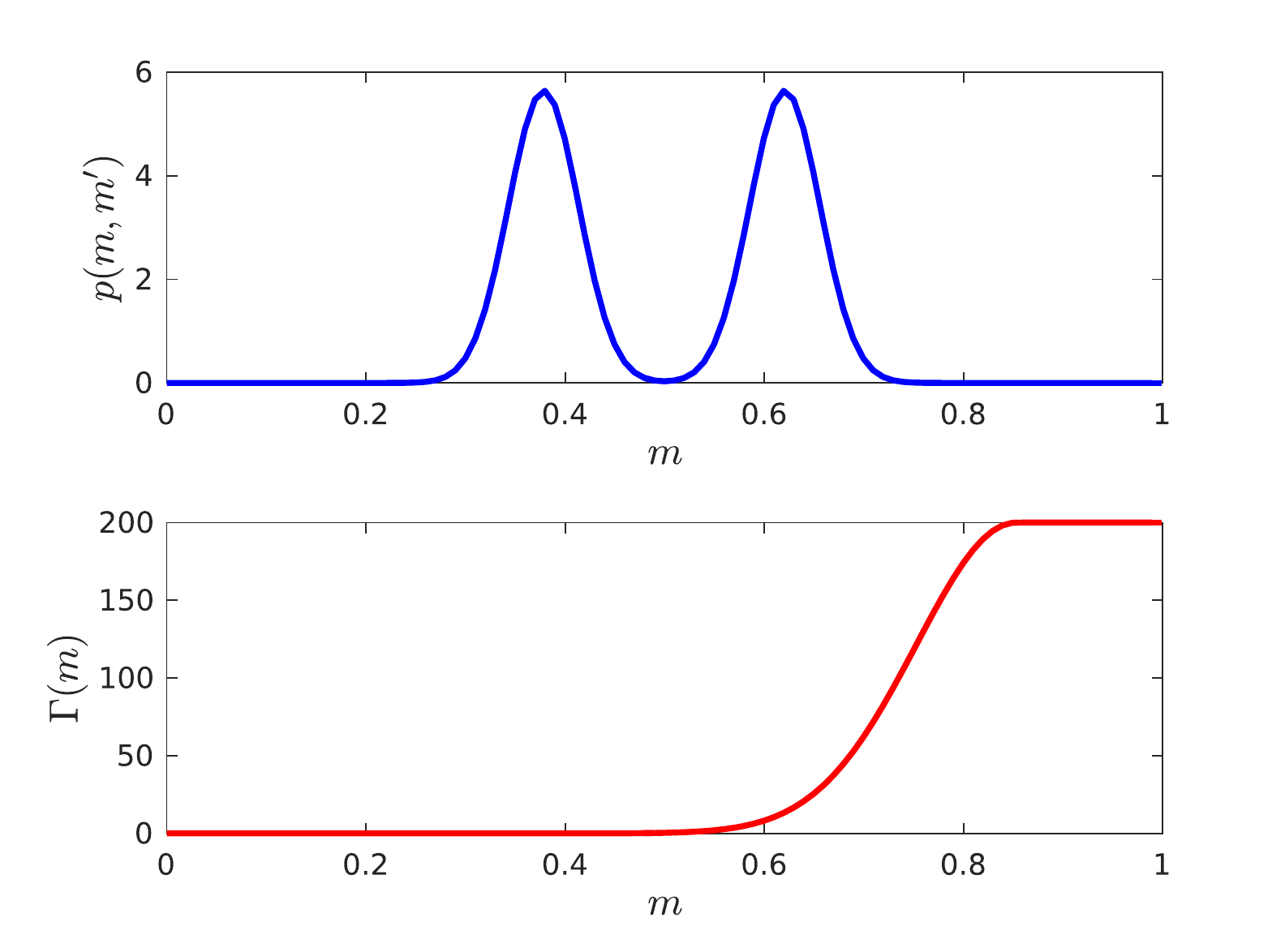}
		\caption{Partitioning function $p(m,m\str)$ and breakage frequency $\Gamma(m)$ for $m_t=0.3784$, $m_d=0.8525$, $\gamma=200$, $\delta=50$, $\lambda=5.6419$, $\beta=400$, fixed $m\sp{\prime}=0.999$ and $m\in[0.001,0.999]$.}
		\label{Fig:plotpartgamma}
	\end{figure}
	The upper graphic in Figure~\ref{Fig:plotpartgamma} shows two peaks. These peaks represent cell division in mother and daughter cell where one of the peaks is centered at $m_t$ and the other peak is centered at $m\str-m_t$. Moreover, the lower graphic in Figure~\ref{Fig:plotpartgamma} shows that beginning from the value of $m_t$ the probability of cell division rises.\\
	The six terms in eq. \eqref{eq:PBE} can be interpreted in the following way. The first term describes the accumulation of cells in time and the second term describes the loss of cells of mass $m$ because they grow into larger cells. Moreover, the third term expresses the birth of cells of mass $m$ resulting from the division of larger cells. Furthermore, the fourth term stands for the loss of cells of mass $m$ due to cell division resulting in the birth of smaller cells. The fifth and sixth term represent the loss of cells due to death, related to a high ethanol concentration and due to other circumstances.\\
	To describe the process of white wine fermentation, in addition to this population balance equation \eqref{eq:PBE}, differential equations for the other substrates are needed. 
	The reaction rates used in the following come from our ODE model introduced in \citet{schenk17}.\\
	The consumption of the nutrient nitrogen is described by the following equation
	\begin{equation}
	\begin{split}
	&\dfrac{dN}{dt}=-k_1\int_{m_{min}}^{m_{max}}r_{\epsilon}(m,N,S,O)W(m,t)dm
	\end{split}
	\end{equation} 
	with the initial condition 
	\begin{equation}
	N(0)=N_0
	\end{equation}
	and, analogously, the consumption of the nutrient oxygen is represented by
	\begin{equation}
	\begin{split}
	&\dfrac{dO}{dt}=-k_4\int_{m_{min}}^{m_{max}}r(m,N,S,O)W(m,t)dm
	\end{split}
	\end{equation}
	with the initial condition
	\begin{equation}
	O(0)=O_{0}.
	\end{equation}
	For the differential equation of oxygen the growth rate $r$ is not dependent on $\epsilon$, such that
	\begin{equation}
	r(m,N,S,O)=\mu_{max}(T)\dfrac{N}{K_N+N}\dfrac{S}{K_{S_1}+S}\dfrac{O}{K_{O}+O}m.
	\end{equation}
	Thereby, $k_1>0$ and $k_4>0$ are the yield coefficient for nitrogen or respectively oxygen.\\
	Moreover, the sugar consumption due to the conversion into alcohol and the consumption for yeast activity described by
	\begin{equation}
	\begin{split}
	&\dfrac{dS}{dt}=-\int_{m_{min}}^{m_{max}}q(m,N,S,E,O)W(m,t)dm
	\end{split}
	\end{equation}
	with the initial condition
	\begin{equation}
	S(0)=S_0.
	\end{equation}  
	The consumption rate $q(m,N,S,E,O)$ is represented by
	\begin{equation}
	q(m,N,S,E,O)=k_2\; q_E(m,S,E)+k_3\;r_{\epsilon}(m,N,S,O).
	\end{equation}
	$k_2>0$ and $k_3>0$ are the yield coefficients associated with the part of sugar converted into alcohol and the part of sugar consumed as a nutrient for the yeast, respectively.
	The accumulation of ethanol is modeled by
	\begin{equation}
	\begin{split}
	&\dfrac{dE}{dt}=\int_{m_{min}}^{m_{max}}q_E(m,S,E)W(m,t)dm
	\end{split}
	\end{equation}
	with the initial condition
	\begin{equation}
	E(0)=0.
	\end{equation} 
	The ethanol accumulation rate $q_E(m,S,E)$ is expressed by
	\begin{equation}
	q_E(m,S,E)=\beta_{max}(T)\dfrac{S}{K_S+S}\dfrac{K_E(T)}{K_E(T)+E}m.
	\end{equation} 
	Thereby, $\mu_{max}$ and $\beta_{max}$ are reaction rates, $K_E$ describes the growth inhibition by ethanol and $K_N$, $K_{O}$, $K_{S_1}$, $K_{S_2}>0$ denote the Michaelis constants for nitrogen, oxygen and sugar.\\Besides, $\mu_{max}$, $\beta_{max}$ and $K_E$ are assumed to be linearly dependent on the temperature $T$, i.e.\ $\mu_{max}(T)=\mu_1 T-\mu_2$ with $\mu_1$, $\mu_2\ge0$ and $\beta_{max}(T)=\beta_1 T-\beta_2$ with $\beta_1$, $\beta_2\ge0$ and $K_E(T)=-K_{E_1} T+K_{E_2}$ with $K_{E_1}$, $K_{E_2}\ge 0$ and with $T$ according to a temperature profile used in practice, a constant temperature for the first half of the fermentation process, a linear increase to a second constant higher temperature that is kept for the second half of the fermentation process. All these rates should be non-negative.
	\subsection{Numerical Scheme}
	\subsubsection{Spatial Discretization: Discretization of PIDE in Mass}\label{Sec:FVM}
	FVMs are popular for the discretization of hyperbolic governing equations. The method makes immediate use of the conservation laws like the law for the conservation of mass. The equations are discretized by division of the space into a finite number of control volumes.
	More information related to the finite volume method in general can be found e.g.\ in \citet{blazekCFD, munzNum, wesseling}.\\ In the following, the finite volume scheme for the discretization in mass of the system of integro-differential equations, introduced in Section~\ref{sec:IDEmod} is derived. One major advantage of this method is that by using it the conservation of mass is guaranteed. Here the position definition of the control volume is described by a cell-centered scheme.
	The integration of the population balance equation \eqref{eq:PBE} over the control volume $\Omega_i=[m_i, m_{i+1}]$ results in  
	\begin{equation}
	\begin{split}
	\int_{\Omega_i}\dfrac{\partial W(m,t)}{\partial t}dm + \int_{\Omega_i}\dfrac{\partial (r_{\epsilon}(m,N,S,O)W(m,t))}{\partial m} dm=\int_{\Omega_i} g_s dm ,\quad\text{where}
	\end{split}
	\label{eq:fveq1}
	\end{equation}
	\begin{equation}
	g_s=2\int_\Omega k(m,m\sp{\prime})W(m\sp{\prime},t)dm\sp{\prime}
	-\Gamma(m)W(m,t)-\Phi(E)W(m,t)-k_dW(m,t)\quad \text{	with}
	\label{eq:fveq2}
	\end{equation}
	\begin{equation}
	k(m,m\sp{\prime})=p(m,m\sp{\prime})\Gamma(m\sp{\prime}).
	\end{equation}
	The cell number density $W(m,t)$ is chosen to be piecewise constant on $\Omega_i$, such that $W(m,t)=w_i(t)$, $\forall m\in[m_i,m_{i+1}]$. An illustration of these piecewise constant approximations can be found in Figure~\ref{fig:fvm}.	\begin{figure}[htbp]
		\centering
			\begin{tikzpicture} 
			\draw[->,thick] (-1,0) -- (8,0);
			\draw[thick] (1.5,-2pt)--(1.5,2pt);
			\draw[thick] (3,-2pt)--(3,2pt);
			\draw[thick] (4.5,-2pt)--(4.5,2pt);
			\draw[thick] (7.5,-2pt)--(7.5,2pt);
			\draw (1.5,-0.6) node{$m_1$};
			\node [position label] (OmegaStart) at (0.0,0.0) {};
			\draw (3,-0.6) node{$m_2$};
			\draw (4.5,-0.6) node{$m_3$};
			\draw (6.,-0.6) node{$\ldots$};
			\node [position label] (OmegaEnd) at (7.5,0.0) {}; \draw (7.5,-0.6) node{$m_{max}$};
			\draw[-,thick] (0.1,0.5) -- (1.4,0.5);
			\draw (0.75,0.7) node{$w_0$};
			\draw (0.75,-0.4) node{$\Omega_0$};
			\draw[-,thick] (1.6,1.) -- (2.9,1.);
			\draw (2.25,1.2) node{$w_1$};
			\draw (2.25,-0.4) node{$\Omega_1$};
			\draw[-,thick] (3.1,0.3) -- (4.4,0.3);
			\draw (3.75,0.5) node{$w_2$};
			\draw (3.75,-0.4) node{$\Omega_2$};
			\draw [thick,brace, decoration={raise=0.3ex}] (OmegaStart.south) -- node [position label,pos=0.5] {$\Omega$} (OmegaEnd.south);
			\draw[->,thick] (0,-1.75) -- (0,2);			
			\end{tikzpicture}\nolinebreak
		\caption{Illustration of piecewise constant approximations for control volumes for FVM.}\label{fig:fvm}
	\end{figure}
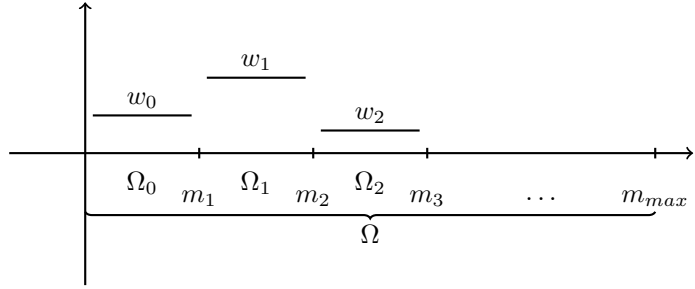The first term in eq.~\eqref{eq:fveq1} becomes	\begin{equation}
	\int_{\Omega_i}\dfrac{\partial W(m,t)}{\partial t}dm=\dot{w_i}(m_{i+1}-m_i).
	\end{equation}
	For the second term, the first order upwind scheme taking the velocity at the right hand side boundary of the cell is used for the approximation of the flux. 
	This yields
	\begin{equation}
	\begin{split}
	&\int_{\Omega_i}\dfrac{\partial (r_{\epsilon}(m,N,S,O)W(m,t))}{\partial m} dm\\&= r_{\epsilon}(m_{i+1},N,S,O)W(m_{i+1},t)-r_{\epsilon}(m_i,N,S,O)W(m_i,t)\\&\approx r_{\epsilon}(m_{i+1},N,S,O)w_{i}(t)-r_{\epsilon}(m_{i},N,S,O)w_{i-1}(t),
	\end{split}
	\end{equation}
	where the velocity or in other words the growth rate $r_{\epsilon}$ is affine linear in $m$. 
	The first part of the source term \eqref{eq:fveq2} yields
	\begin{equation}
	\begin{split}
	&2\int_{\Omega_i}\int_\Omega k(m,m\sp{\prime})W(m\sp{\prime},t)dm\sp{\prime}dm\\&=
	2\sum_{j=0}^{N_W}\int_{\Omega_i}\int_{\Omega_j}
	k(m,m\sp{\prime})W(m\sp{\prime},t)dm\sp{\prime}dm\\&=
	2\sum_{j=0}^{N_W}w_j(t)\underbrace{\int_{\Omega_i}\int_{\Omega_j}
		k(m,m\sp{\prime})dm\sp{\prime}dm}_{=:K_{ij}}\\
	&=2\sum_{j=0}^{N_W}w_j(t)K_{ij}
	\end{split}
	\end{equation}
	and the second part of the source term \eqref{eq:fveq2} becomes 
	\begin{equation}
	\int_{\Omega_i}\Gamma(m)W(m,t)dm=
	w_i(t)\int_{\Omega_i}\Gamma(m)dm
	\end{equation}
	and both death terms can be written in the following way as 
	\begin{equation}
	\int_{\Omega_i}k_dW(m,t)dm=
	w_i(t)(m_{i+1}-m_i)\quad\text{and}
	\end{equation} 
	\begin{equation}
	\int_{\Omega_i}\Phi(E)W(m,t)dm=
	w_i(t)(m_{i+1}-m_i).
	\end{equation}
	All in all, equation~\eqref{eq:PBE} in discretized form looks like the following
	\begin{equation}
	\begin{split}
	\dot{w_i}=&\dfrac{1}{m_{i+1}-m_i}\bigg[-( r_{\epsilon}(m_{i+1},N,S,O)w_{i}(t)-r_{\epsilon}(m_{i},N,S,O)w_{i-1}(t))\\&+ 2\sum_{j=0}^{{N_W}}\left(w_j(t)\int_{\Omega_{i}}\int_{\Omega_{j}}p(m,m^{\prime})\Gamma(m^{\prime})dm^{\prime}dm\right)-w_i(t)\int_{\Omega_i}\Gamma(m)\;dm\bigg]\\& -\Phi(E)w_i(t)-k_dw_i(t),\quad i=1,2,\ldots,N_W-1,
	\end{split}
	\label{eq:fvfinal}
	\end{equation}
	where $N_W$ is the number of cells used in the finite volume scheme.
	Additionally, we have the following boundary conditions 
	\begin{equation}\begin{aligned}
	&\frac{1}{\Delta m}r_{\epsilon}(m_{0},N,S,O)w_0(t)=0=\frac{1}{\Delta m}r_{\epsilon}(m_{N_W},N,S,O)w_{N_W}(t).
	\end{aligned}
	\label{eq:fvbound}
	\end{equation} 
	To describe the process of wine fermentation, the whole system of equations introduced in Section~\ref{sec:IDEmod}, needs to be put in discretized form. Thus, the ordinary integro-differential equations for the product and substrates concentration development have to be discretized and added to \eqref{eq:fvfinal} and \eqref{eq:fvbound}. The application of the finite volume discretization for these equations yields the following.\\ 
	For the sugar consumption, it follows
	\begin{equation}
	\frac{dS}{dt}=-\sum_{i=1}^{{N_W}-1} \tilde{q}(N,S,E,O)\left(\dfrac{m_{i+1}+m_i}{2}\right)w_i(m_{i+1}-m_i),
	\end{equation}
	where $q=\tilde{q}(N,S,E,O)m$.\\
	The accumulation of ethanol results in
	\begin{equation}
	\frac{dE}{dt}=\sum_{i=1}^{{{N_W}-1}} \tilde{q}_E(S,E)\left(\dfrac{m_{i+1}+m_i}{2}\right)w_i(m_{i+1}-m_i),
	\end{equation}
	where $q_E=\tilde{q_E}(S,E)m$.\\
	For the nutrients, this results in
	\begin{equation}
	\dfrac{dN}{dt}=-k_1\sum_{i=1}^{{{N_W}-1}}\tilde{r}_{\epsilon}(N,S,O)\left(\dfrac{m_{i+1}+m_i}{2}\right)w_i(m_{i+1}-m_i)
	\end{equation} 
	for the nitrogen consumption and in 
	\begin{equation}
	\dfrac{dO}{dt}=-k_4\sum_{i=1}^{{{N_W}-1}}\tilde{r}(N,S,O)\left(\dfrac{m_{i+1}+m_i}{2}\right)w_i(m_{i+1}-m_i)
	\end{equation}

	for the oxygen consumption, where as above $r_{\epsilon}(m,N,S,O)=\tilde{r_{\epsilon}}(N,S,O)m$ and \\${r(m,N,S,O) = \tilde{r}(N,S,O)m}$. 
	This yields the following system of differential equations 	
	\begin{equation}
	\dot{y}=f(t,y(t)), \text{ where } y=(w_i, N, E, S, O)^T\; \text{ for }i=0,\ldots,N_W.
	\label{eq:fullsys}
	\end{equation}
	\subsubsection{Temporal Discretization}\label{sec:tempdiscret}
	To receive a numerical solution to eq.~\eqref{eq:fullsys}, it still needs to be discretized in time. There are a lot of different methods available for solving such a system and there are several advantages and disadvantages for the use of certain methods. 
	In the following, the implicit trapezoidal rule will be used. 
	The time steps are chosen with respect to the Courant-Friedrichs-Lewy condition for explicit methods as e.g.\ in \citet{blazekCFD,wesseling}. The CFL condition is an informative criterion related to the relationship in between the product of the velocity and the time step to here the mass interval length.\\
	The implicit trapezoidal rule is a second order method and A-stable. Discretization with these schemes yields a system of nonlinear equations. This can be solved using Newton's method.
	\begin{remark}\label{rem:cnNLP}
		The implicit trapezoidal rule for PIDE as in Section~\ref{sec:IDEmod} which has to be discretized regarding space and time results in
	 \begin{equation}g(y_i^{n+1})=y_i^{n+1}-y_i^{n}-\frac{h}{2}( f(t_i^{n+1},y_i^{n+1})+f(t_i^n,y_i^n))=0\end{equation} 	 with $n$ denoting the time iterate and $i$ the space iterate. Newton's method applied to this, results in
		\begin{align*}
		&y_i^{n+1}=y_i^n-{\J_g^{-1}(y_i^n)}g(y_i^n)\\&=y_i^n-\left(I-\frac{h}{2\Delta m}{\J_f(y_i^n)}\right)^{-1} (y_i^n-y_i^{n-1}-\frac{h}{2\Delta m}(f(y_i^n)+ f(y_i^{n-1}))).
		\end{align*}
	\end{remark}
	More details for the methods used here and related stability and convergence results can be found e.g.\ in \citet{hairer2010solvingII, deuflhard2008gewoehnliche, stoer2006numerische, plato2006numerische}.
	\subsection{Existence and Uniqueness of the Solution}
	In the following, existence and uniqueness of the solution of a simplified version of the IDE model, introduced in Section~\ref{sec:IDEmod}, is studied. For this simplified case, the growth rate $r_{\epsilon}$ is given
	, such that we mainly consider the population balance equation \eqref{eq:PBE} with its initial and boundary conditions \eqref{eq:PBEIC} and \eqref{eq:PBEBC} but a different $r_{\epsilon}$ in that case. Let us first reformulate the equation introduced in \eqref{eq:PBE}. 
	In this paper, results for a simplified case with constant velocity and constant substrate concentrations are presented.
	In detail, here the other substrates like sugar, nitrogen, oxygen and ethanol are assumed to be constant, such that $r_{\epsilon}$ is given by ${\bar{r}}$. Therefore, in the following ${\bar{r}}$ refers to $r(m,N,S,O)$ and $r_{\epsilon}(m,N,S,O)$ but with constant substrate concentrations and not dependent on $m$. Then with the classification concept for first order PDEs \citep{prasad1985partial, hellwig1977, smoller1994}, 
	the equation can be classified as a semilinear hyperbolic partial integro-differential equation. Investigations for the more complex semilinear case with linear velocity in $m$ and the quasilinear case are currently in progress.\\
	This equation studied here is weakly hyperbolic because its characteristic polynomial has exactly one real distinct eigenvalue which is represented by $\bar{r}$.\\
	What follows in this section is mainly based on \citet{dautray5, dautray6}. In general, an approach based on semigroup theory is used.\\
	With the velocity term $\bar{r}$ given as a constant, the considered problem in this section is represented by
	\begin{subequations}\label{eq:probexun1}
	\begin{align}
		&\Parttwo{W(m,t)}{t}+\bar{r}\cdot \div(W(m,t))+\Sigma(m) W(m,t)=K W(m,t)\label{eq:probexun1a},\\&\hspace*{2em}\; m\in M, \;\bar{r}\in\R^+, \;t>0\notag\\
		& W\big|_{\Theta}=0,\; \Theta=\{m_{min},m_{max}\}\label{eq:probexun1b}
		\\& W(m,0)=W_0\text{ on }M,\; W_0\text{ given.}\label{eq:probexun1c}
		\end{align}
	\end{subequations}
	As $E$ is constant $\Phi(E)$ is constant as well, such that we can rewrite the death terms for this case as:
	\begin{equation}
	\bar{k}_d:=\Phi(E)+k_d
	\end{equation}
	$\Sigma$ is a positive function of $m$ with 	 
	\begin{equation}
	\Sigma(m)=\Gamma(m)+\bar{k}_d 
	\end{equation}	 
	and the given operator $K$ is given by	 
	\begin{equation}
	(KW)(m)=\int_{M} f(m,m\sp{\prime})W(m\sp{\prime},t)dm\sp{\prime}\quad\text{	with}
	\label{eq:operK}
	\end{equation}	  
	\begin{equation}
	f(m,m\sp{\prime})=p(m,m\sp{\prime})\Gamma(m\sp{\prime})
	\end{equation}	 
	a given positive function that is measurable with respect to $m$ and $m\sp{\prime}$. Moreover, $M$ is represented by $M:=(m_{min},m_{max})$. 
	In the following, $L^2(M)$ is always assumed to be a real Banach space with the norm $x\to\|x\|_{L^2(M)}$ and $\{G(t)\}_{t\ge0}$ a semigroup of class $\Co$ over $L^2(M)$. In general for $x\in L^2(M)$, the function $t\to G(t)x$ is not differentiable unless $x \in D(A)$, where $D(A)$ as in Definition~\ref{Def:setdiffvec}. Let $\LBA(L^2(M))$ be the vector space of continuous linear mappings of $L^2(M)$ into $L^2(M)$.\\
	Before we start with the existence and uniqueness investigations, let us first clarify some terms and definitions regarding semigroup theory according to \citet{dautray5}.  
	\begin{definition}{Semigroup of class $\Co$}\\
		Let $\{G(t)\}_{t\ge0}$ be a family of elements $G(t)\in\LBA(L^2(M))$ for $t\ge0$. This family forms a semigroup of class $\Co$ in $L^2(M)$ if it fulfills these conditions
		\begin{align}
		\begin{cases}
		\quad G(s+t)=G(s)G(t)\quad \forall s,t\ge 0 &(i) \quad\text{(algebraic property)}\\
		\quad G(0) = Id &(ii) \quad\text{(identity in $\LBA(L^2(M))$)}\\
		\quad \underset{t\to +0}{\lim}\|G(t)x-x\|_{L^2(M)}=0 \quad\forall x \in L^2(M) &(iii) \quad \text{(topological property)}.
		\end{cases}
		\label{eq:semigroupprop}
		\end{align}		 
		\label{Def:semigroup}
	\end{definition}
	\begin{definition}{Set of differentiable vectors}\\
		We call $D(A)$ the set of differentiable vectors in $L^2(M)$, i.e.\ the subset of elements $x\in L^2(M)$ such that the function $t\to G(t)x$ is differentiable for $t\ge0$. 
		Because of the algebraic property \eqref{eq:semigroupprop}(i), $D(A)$ is represented by 
		\begin{equation}
		D(A)=\{x\in L^2(M); \dfrac{G(h)x-x
		}{h}\text{ converges in } L^2(M)\text{ as }h\to+0\}.
		\label{eq:DA}
		\end{equation}	 
		From now on, let $A_h$ be an operator defined by	 
		\begin{equation}
		A_h:= \dfrac{G(h)-Id}{h}
		\end{equation}
		with $A_h\in\LBA(L^2(M))\quad \forall h>0$.
		\label{Def:setdiffvec}
	\end{definition}
	\begin{definition}{Infinitesimal generator of a semigroup}\\
		An operator $A$ defined as a linear mapping from $D(A)$ into $L^2(M)$, precisely as $$\underset{h\to+0}{\lim}A_hm=Am$$ with $D(A)$ as in 
		\eqref{eq:DA}, is called the infinitesimal generator of the semigroup $\{G(t)\}_{t\ge0}$.
	\end{definition}
	Let $A$ be the unbounded operator in $L^2(M)$ defined by 		 
	\begin{equation}
	\begin{cases}
	&(AW)(m)=-\bar{r}\cdot \div(W(m,t))\\
	& D(A)=\{W\in L^2(M);\; AW\in L^2(M), W\big|_{\Theta}=0\}. 
	\end{cases}
	\label{eq:defADA}
	\end{equation}
	Then, $A$ is called advection operator.
	Problem~\eqref{eq:probexun1} is equivalent to 
	\begin{equation}
	\begin{cases}
	&\Parttwo{W}{t}=TW\\& W(0)=W_0 \quad\text{with}
	\end{cases}
	\label{eq:probexun1prime}
	\end{equation} 
	\begin{equation}
	T=A-\Sigma(m)+K
	\end{equation}		 
	where $K$ is an integral operator, defined by \eqref{eq:operK}, which is bounded in $L^2(M)$ under certain assumptions for the kernel $f$.
	In order to solve problem~\eqref{eq:probexun1prime}, we first have to determine the semigroup generated by the operator $T$. This semigroup should be a semigroup of class $\Co$ in $L^2(M)$.
	\begin{prop}		\label{Prop:Prop2}
		Let $A$ be the infinitesimal generator of a semigroup of class $\Co$ in $L^2(M)$, $\Sigma\in L^{\infty}(M)$ be a given function and $K$ a continuous linear operator from $L^2(M)$ into $L^2(M)$, then the operator  
		\begin{align*}
		\begin{cases}
		\quad T=A-\Sigma(m)+K &(i)\\
		\quad D(T)=D(A) &(ii)
		\end{cases}
		\end{align*}		 
		is the infinitesimal generator of a semigroup of class $\Co$ in $L^2(M)$. If further $K$ and the semigroup generated by $A$ operate in the positive cone of functions of $L^2(M)$ or respectively $L^2(M)$, then the semigroup generated by $T$ operates in the cone of positive functions of $L^2(M)$. 
	\end{prop}
	\begin{proof}
		The proof can be found in \citet{dautray6}.
	\end{proof}
	\begin{remark}
		The family $\{G(t)\}_{t\ge0}=e^{tA}$ forms a semigroup in $L^2(M)$ with infinitesimal generator $A$ and the family $\{G_1(t)\}_{t\ge0}=e^{t(A-\Sigma+K)}$ forms a semigroup in $L^2(M)$ with infinitesimal generator $T$.
	\end{remark}
	\begin{lemma}\label{lem:fbounded}
		Let $f(m,m\sp{\prime})$ be a given real positive function ($f\ge0$) and  measurable with respect to $m$ and $m\sp{\prime}$. Then there exist positive constants $C_a$ and $C_b$ such that	 
		\begin{equation}
		\begin{cases}
		&\int_{M} f(m,m\str)\;dm \le C_a \quad \forall m\sp{\prime}\in {M}\quad \text{(a)}\\
		&\int_{M} f(m,m\str)\;d{m\str} \le C_b \quad \forall m\in {M}\quad \text{(b)}
		\end{cases}
		\end{equation}
	\end{lemma}
	\begin{proof}
		(a) First, let 	 
		\begin{equation}
		\int_{M}f(m,m\str) dm=2\int_{M}p(m,m\str)\Gamma(m\str)dm.
		\end{equation}	 
		Due to the structure of $p(m,m\str)$ and $\Gamma(m\str)$ we distinguish between three different cases. \begin{enumerate}
			\item $m\str>m$ and $m\str>m_t$ and $m\str>m_d$:
			Then, we have 
			\begin{equation}
			\begin{aligned}
			\int_{M}f(m,m\str) dm&=2\int_{M}p(m,m\str)\Gamma(m\str)dm\\&=2\int_{M}(\lambda e^{-\beta(m-m_t)^2}+\lambda e^{-\beta(m-m\str+m_t)^2})\gamma dm \\&=2\lambda\gamma \int_{M}( \underbrace{e^{-\beta(m-m_t)^2}}_{\le 1}+ \underbrace{e^{-\beta(m-m\str+m_t)^2}}_{\le 1}) dm\\&\le 2\lambda\gamma \int_{M} 2 dm\\&=\lambda \gamma 4m\big|_{M}\\&<\lambda \gamma 4 m_{max}=:C_{a_1}
			\end{aligned}
			\notag
			\end{equation}
			with ${M}=(m_{min},m_{max})$.
			\item $m\str>m$ and $m\str>m_t$ and $m_t<m\str<m_d$:
			It holds
			\begin{equation}
			\begin{aligned}
			\int_{M}f(m,m\str) dm&=2\int_{M}p(m,m\str)\Gamma(m\str)dm\\&=2\int_{M}(\lambda e^{-\beta(m-m_t)^2}+\lambda e^{-\beta(m-m\str+m_t)^2})\gamma e^{-\delta(m\str-m_d)^2} dm \\&=2\lambda\gamma \int_{M}(\underbrace{e^{-\beta(m-m_t)^2}}_{\le 1}+ \underbrace{e^{-\beta(m-m\str+m_t)^2}}_{\le 1})\underbrace{e^{-\delta(m\str-m_d)^2}}_{\le 1} dm\\&\le 2\lambda\gamma \int_{M} 2 dm\\&=\lambda \gamma 4m\big|_{M}\\&<\lambda \gamma 4 m_{max}=:C_{a_2}
			\end{aligned}
			\notag
			\end{equation}
			with ${M}=(m_{min},m_{max})$.
			\item else: We obtain	 
			\begin{equation}
			\int_{M} 0\;dm = const=:C_{a_3}.
			\notag
			\end{equation}	 
		\end{enumerate}
		(b) The cases for	 
		\begin{equation}
		\int_{M}f(m,m\str) dm\str=2\int_{M}p(m,m\str)\Gamma(m\str)dm\str
		\notag
		\end{equation} 
		can be shown analogously to (a). 
	\end{proof}Now, all the preconditions for the following lemma are given. 
	\begin{lemma}
		Let Lemma~\ref{lem:fbounded} hold.
Then the operator $K$ defined by $$(K\eta)(m)=\int_{M} f(m,m\str)\eta(m\str)dm\str \quad \forall \eta \in L^2(M)$$ is linear and continuous from $L^2(M)$ into $L^2(M)$. 
		\label{Lem:bounded}
	\end{lemma}
	\begin{proof}
		The proof works analogously to the proof of Lemma~1 in \citet[Chapter~XXI,\S2]{dautray6} with the special case $p=p\str=2$, i.e.\ Cauchy-Schwarz inequality.
	\end{proof}
	Let $\W_2$ be the space defined by 
	\begin{equation}
	\W_2=\{W\in L^2(M): \bar{r} \cdot\div W\in L^2(M)\}
	\end{equation}
	\begin{remark}
		$W$ is called a weak solution of \eqref{eq:probexun1} if $y\in\W_2$, $W(0)=W_0$ and \begin{equation}
		\begin{aligned}
		&\frac{d}{dt}\int_{M}v(m)W(m,t)dm+\int_{M}(\bar{r}W(m,t))\cdot \div(v(m))dm\\&=-\int_{M}KW(m,t)v(m)dm+\int_{M}(\Sigma(m)+\tilde{\bar{r}})W(m,t)v(m)dm\quad\forall v\in H^1({M}),
		\end{aligned}
		\end{equation}
		which can be derived making use of Gauss's theorem or integration by parts in multidimensions respectively as e.g. in \citet{forster} and $W(m,t)= 0$ on $\Theta$. 
	\end{remark}
	The solution of problem~\eqref{eq:probexun1} is given by 
	\begin{theorem}
		Let the data of problem~\eqref{eq:probexun1} satisfy $\Sigma \in L^{\infty}(M)$ with $\Sigma(m) \ge 0 \;\forall m$, $K$ be the operator defined in Lemma~\ref{Lem:bounded}, where $f$ is a positive function in terms of Lemma~\ref{lem:fbounded} and $W_0\in L^2(M)$.\\
		Then, problem~\eqref{eq:probexun1} has a unique weak solution $W$ in the space $\W_2$ and \(W \in \C([0,t_f]; L^2(M))\).
		If for $W_0$ it holds also that $\bar{r}\div(W_0)\in L^2(M)$ and $W_0\big|_{\Theta}=0$ $(W_0 \in D(A))$, then $W$ is a strong solution of problem~\eqref{eq:probexun1}.\\ This solution satisfies $W\in\C([0,t_f];L^2(M))$, $\bar{r}\cdot\div(W(m,t))\in\C([0,t_f];L^2(M))$ and $W(t)\big|_\Theta=0 \; \forall t\in[0,t_f]$ \quad$(W\in \C([0,t_f],D(A))).$\\
		Furthermore, with $W_0\ge0$ we have $W\ge0$. 
		\label{Theor:weakstrongsol}
	\end{theorem} 
	\begin{proof}
		See below.
	\end{proof}
	Before we prove this theorem, let us first show a necessary precondition for Theorem~\ref{Theor:weakstrongsol}.
	\begin{prop}
		$\Sigma\in L^{\infty}(M)$.
	\end{prop}
	\begin{proof}
		\begin{equation}
		\esssup\|\Sigma\|_{L^{\infty}({M})}=\inf\{C\ge 0:|\Sigma(m)|\le C\text{ for almost every } m\},\quad \text{where}
		\notag
		\end{equation} 
		\begin{equation}
		\begin{aligned}\quad|\Gamma(m)+\bar{k}_d|&\le|\Gamma(m)|+\bar{k}_d\\&<\gamma +\bar{k}_d=:C_1\ge 0\; \text{ and}\;<\infty \quad\forall m,
		\end{aligned}
		\notag
		\end{equation} 	 
		Then	 
		\begin{equation}
		\inf\{C\ge 0:|\Sigma(m)|\le C\text{ for almost every } m\}<\infty
		\notag
		\end{equation}	 
		with $C<C_1$ holds.
	\end{proof}
	\begin{proof}[Proof of Theorem \ref{Theor:weakstrongsol}]
		We apply Proposition~\ref{Prop:Prop2} with $A$ defined by 
		\begin{equation}
		AW = \bar{r}\div (W(m,t)) 
		\end{equation}

		and 	\begin{equation}
		D(A) = \{W \in \W^2(M);W|_{\Theta}=0\},
		\end{equation}
		with $\W^2(M)=\{W\in L^2(M); \;\bar{r} \cdot\div W \in L^2(M)\}$ and set
\begin{equation}
		W(t)=e^{t(A+K-\Sigma)}W_0.
		\end{equation}
		$W$ is a weak solution of \eqref{eq:probexun1}. Furthermore, we will also demonstrate that this is a solution in the sense of distributions (i.e.\ $\in \mathcal{D}\str(M\times(0,t_f))$ of \ref{eq:probexun1a}).\\
		Let us now show the uniqueness of our solution. Therefor suppose that $W_0=0$ and that $W\in\W_2$ fulfills \eqref{eq:probexun1}. The application of formula~2.33 from \citet[Chapter~XXI,\S2]{dautray6} to $W$ (with $w=W$ and $u=W$) results in
			  
		\begin{equation}\begin{aligned}
		&\int_{0}^{\tau}\left(\left(W(t),\bar{r} \div W(t)+\Parttwo{W}{t}(t)\right)_{L^2(M)}+\left(\bar{r} \div W(t)+\Parttwo{W}{t}(t),W\right)_{L^2(M)}\right)dt\\&=
		2\int_0^{\tau} \left(W(t), KW(t)-\Sigma W(t)\right)_{L^2(M)}dt\\&=\|W(\tau)\|_{L^2(M)}^2 +\underbrace{\int_0^{\tau}(\|W(t)\big|_{\Theta}\|_{L^2(\Theta)}^2) dt}_{=:const}-\underbrace{\|W(0)\|_{L^2(M)}^2}_{=0}\\&\ge \|W(\tau)\|_{L^2(M)}^2,
		\end{aligned}
		\notag
		\end{equation}
			 
		where $const$ is a positive constant.\\
		Moreover, with $\Sigma$ being a positive function, we receive
			 
		\begin{equation}
		\begin{aligned}
		\|W(\tau)\|_{L^2(M)}^2&\le 2\int_0^{\tau}\left(W(t),KW(t)-\Sigma(m)W(t)\right)_{L^2(M)} dt\\&\le 2\|K\|_{L^2(M)} \int_0^{\tau}\|W(t)\|_{L^2(M)}^2 dt.
		\end{aligned}
		\end{equation}
			 
		With Gronwall's lemma (see e.g.\ Chapter XVIII, \S5 in \citet{dautray5}), we get $W=0$.
	\end{proof}
	
	\section{Numerical Results and Discussion}\label{sect:resultsdiscuss}
	To solve the nonlinear system of equations as explained in Remark~\ref{rem:cnNLP}, the derivatives of the right hand side are required. These could be approximated e.g.\ via finite difference schemes. Nevertheless providing gradient information can enhance the algorithm and is recommended to be made available for the discretization scheme.\\ 
	For the numerical results presented in the following, the analytically derived Jacobian of the right hand side $f$ of the system of integro-differential equations as in Appendix~\ref{app:Jacobianf} of this article is used. The system is given by
		   
	\begin{equation}
	\Parttwo{y}{t}=f(y,m,t)
	\end{equation}
		 
	with $f(y,m,t)=(f_{w_i}, f_N, f_E, f_S, f_O)^T \;\forall i=1,\ldots, N_W$ with $y=(w_i,N,E,S,O)^T$ in discretized form. 
	As explained in Section~\ref{Sec:FVM} for the discretized form using a first order upwind scheme for the flux approximation, it is distinguished between $w$ at cell $i$ and cell $i-1$. 
	Note that the right hand side $f$ here also includes the advection term.
	All of the results presented in this section were computed in MATLAB \citep{MATLAB2017b}, with the MATLAB backslash operator as the underlying linear system solver.\\
	In the following first some clarifications for the cell mass are given. 
	The mass of cells measured in experiments available in literature differs from $5\times10^{-7}$mg (C. N\"ageli) to $5.5\times10^{-8}$mg (M. Rubner) to  $6.3-37\times 10^{-8}$mg (G. Seliber and R. Katznelson) \citep{Seliber1929}.
	For the comparison with the model based on ordinary differential equations, an average cell mass of $5\times10^{-7}$mg (C. N\"ageli; \citep{Seliber1929}) is used.
	In this work, cell masses of $m\in[0,1\times10^{-9}$g$]$ are considered.
	For computational purposes this is scaled to $m\in[0.001,0.999]$, such that with the values for the allowed cell masses and $m_t$ and $m_d$ from \citet{dynamichens}, this yields $m_t=3.7917\times10^{-10}$ and $m_d=8.5417\times10^{-10}$ for $m\in[0,1\times10^{-9}$g$]$ and furthermore, $m_t=0.3784$ and $m_d=0.8525$ for $m\in[0.001,0.999]$.
	In detail for $m\in[0,12\times10^{-13}$g$]$, \citet{dynamichens} use $m_{t_0}=4.55\times10^{-13}$g and $m_{d_0}=10.25\times10^{-13}$g, such that $m_t$ and $m_d$ are chosen based on these values combined with a dependence on the effective substrate concentration. Here a simplification is used and the values of $m_{t_0}$ and $m_{d_0}$ are chosen as orientation values for $m_t$ and $m_d$. 
	This results in
	 \begin{equation}m_{t_{nn}}=\frac{(10^{-9}-0)}{12\times10^{-13}-0}(4.55\times10^{-13}-0)+0=3.7917\times10^{-10},\end{equation}  where $m_{t_{nn}}$ denotes the transient mass normalized to $[0,1\times10^{-9}]$,
	and further this yields
	 \begin{equation}m_{t_n}=\frac{(0.999-0.001)}{10^{-9}-0}(3.7917\times10^{-10}-0)+0=0.3784,\end{equation} where $m_{t_n}$ represents the transient mass normalized to $[0.001,0.999]$.\\
	In the same way, this yields
	 
	\begin{equation}m_{d_{nn}}=\frac{(10^{-9}-0)}{12\times10^{-13}-0}(10.25\times10^{-13}-0)+0=8.5417\times10^{-10},\end{equation} where $m_{d_{nn}}$ denotes the division mass normalized to $[0,1\times10^{-9}]$, 
	and further this results in
	 \begin{equation}m_{d_n}=\frac{(0.999-0.001)}{10^{-9}-0}(8.5417\times10^{-10}-0)+0=0.8525,\end{equation} where $m_{d_n}$ represents the division mass normalized to $[0.001,0.999]$.\\
	For all results computed in this section, we use the finite volume discretization with an upwind scheme for the flux approximation as derived in Section~\ref{Sec:FVM}. The involved integrals are approximated by the composite trapezoidal rule with thirty subintervals.\\
	The kinetic parameters for this model and the model based on ordinary differential equations, to which it will be compared to later in this section, according to Table~\ref{Tab:parameterIDE}.
	\begin{table}[h]\centering
		\begin{minipage}[b]{.6\linewidth}
			\centering
		\begin{tabular}{|l S|l S|}\hline
			\text{Parameters}& \text{set} & \text{Parameters}& \text{set}\\\hline
			$\mu_1$ & 0.1681 & $\beta_1$ & 0.1348\\
			$\mu_2$ & 0.0 & $\beta_2$ & 0.0\\
			$K_N$ & 0.1096 & $k_{d_1}$ & 99.86 \\
			$k_1$ & 0.018 & $k_{d_2}$ & 0.0021\\ 
			$K_{S_1}$ & 29.5 & $K_O$ & 0.0007\\
			$K_{S_2}$ & 4.3262 & $k_4$ & 0.0006\\
			$K_{E_1}$ & 0.2616 & $tol$ & 70 \\
			$K_{E_2}$ & 38.90 & $k_d$ & 0.01\\
			$\epsilon$ & 0.02&&\\
			\hline
		\end{tabular}
		\caption{Kinetic parameter values for the IDE model and the ODE model.}	\label{Tab:parameterIDE}
	\end{minipage}
\hspace{.75em}
\begin{minipage}[b]{.3\linewidth}
		\begin{tabular}{|l S|}\hline
			\text{Parameters}& \text{set} \\\hline
			$\gamma$ & 200\\ 
			$\delta$ & 50\\
			$\lambda$ & 5.6419\\
			$\beta$ & 400\\
			\hline
		\end{tabular}
		\caption{Cell division parameter values for the IDE model.}
		\label{Tab:parameterIDEdiv}
\end{minipage}
	\end{table}
	Moreover, the parameters related to yeast cell division in the IDE model are set to the values, given in Table~\ref{Tab:parameterIDEdiv},
	where the parameters $\gamma$, $\delta$ and $\beta$ were set to values based on experience geared to literature values. Then, $\lambda$ is calculated based on these parameter values ensuring that  \begin{equation}\int_Mp(m,m\str)\;dm=1\end{equation} holds. This property assures that the partition probability density function or partitioning function is truly a density function as explained in Section~\ref{sec:IDEmod}.\\
	For the temperature, a traditional fermentation temperature profile as used for example in some fermentation processes performed in 2011 at the DLR (Dienstleistungszentrum L\"andlicher Raum) Mosel in Bernkastel-Kues, one of our public research partners in the R\OE NOBIO project, is used. This means that the temperature is set to $15^{\circ}$C for the first half of the fermentation and then it is increased to $18^{\circ}$C for the second half of the fermentation, where a linear increase in between is assumed here. This profile is illustrated for instance in the last subplot in Figure~9 in blue.\\
	All of the numerical results presented in the following were generated on a 64bit Dell XPS 13 7390 Laptop with Intel(R) Core(TM) i7-10510U CPU at 1.8 GHz with 16.1 GB of RAM.\\ First, the cell number densities for a mass discretization with $30$, $50$, $100$ and $150$ mass cells using the implicit trapezoidal rule for the time discretization are compared. The time step for the different cases is selected with respect to the CFL condition and with respect to the convergence of the Newton's method. The convergence tolerance is set to $1\times10^{-10}$ and the maximum number of iterations is set to $100$ iterations for the Newton's method.
	\begin{figure}[htbp]
					\centering
		\begin{subfigure}[c]{0.49\textwidth}
			\includegraphics[width=17em]{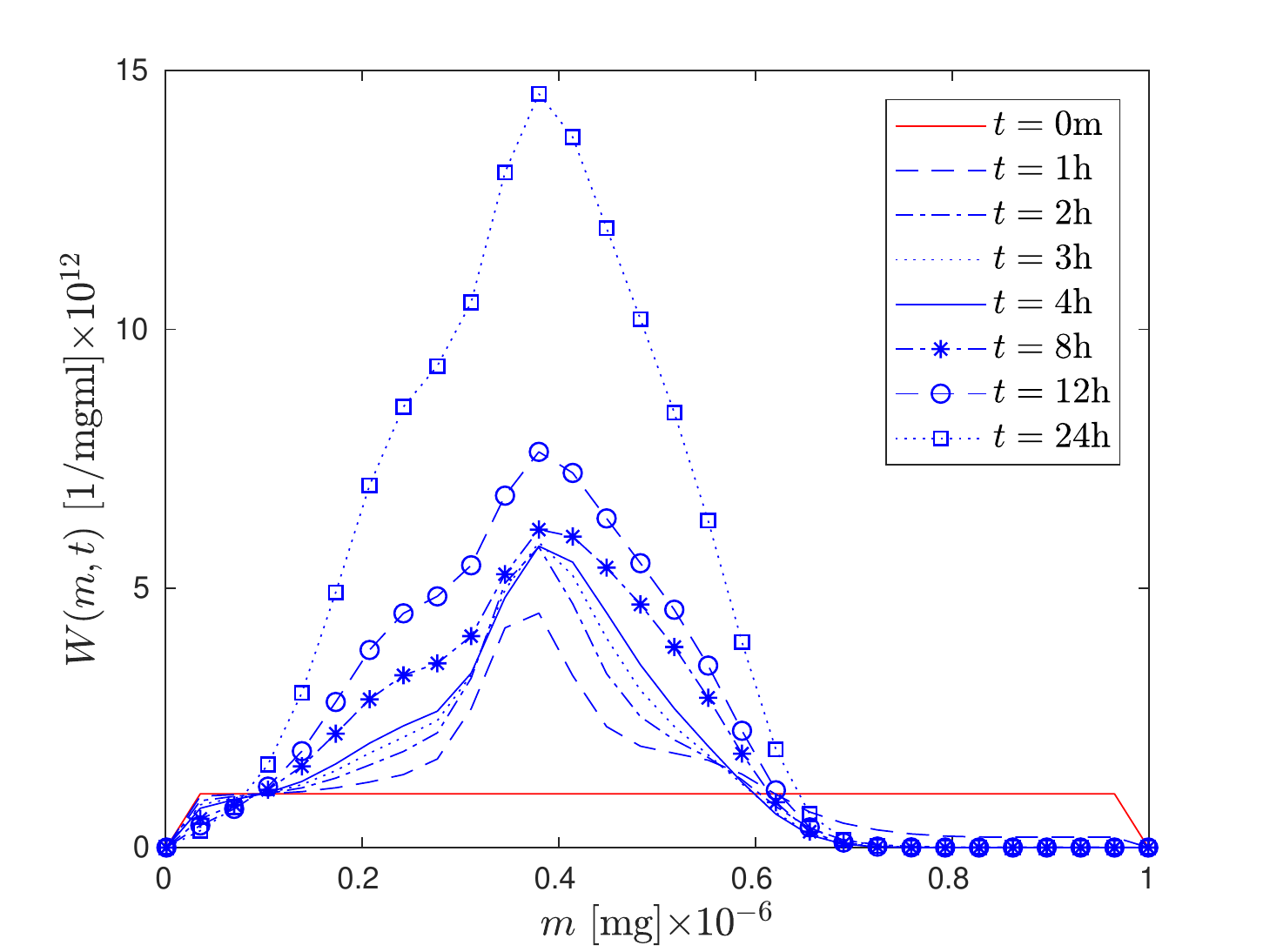}
			\caption{Mass discretization with 30 mass cells and time discretization with $h=\dfrac{1}{48}\approx0.0208$.}
		\end{subfigure}
		\begin{subfigure}[c]{0.49\textwidth}
			\includegraphics[width=17em]{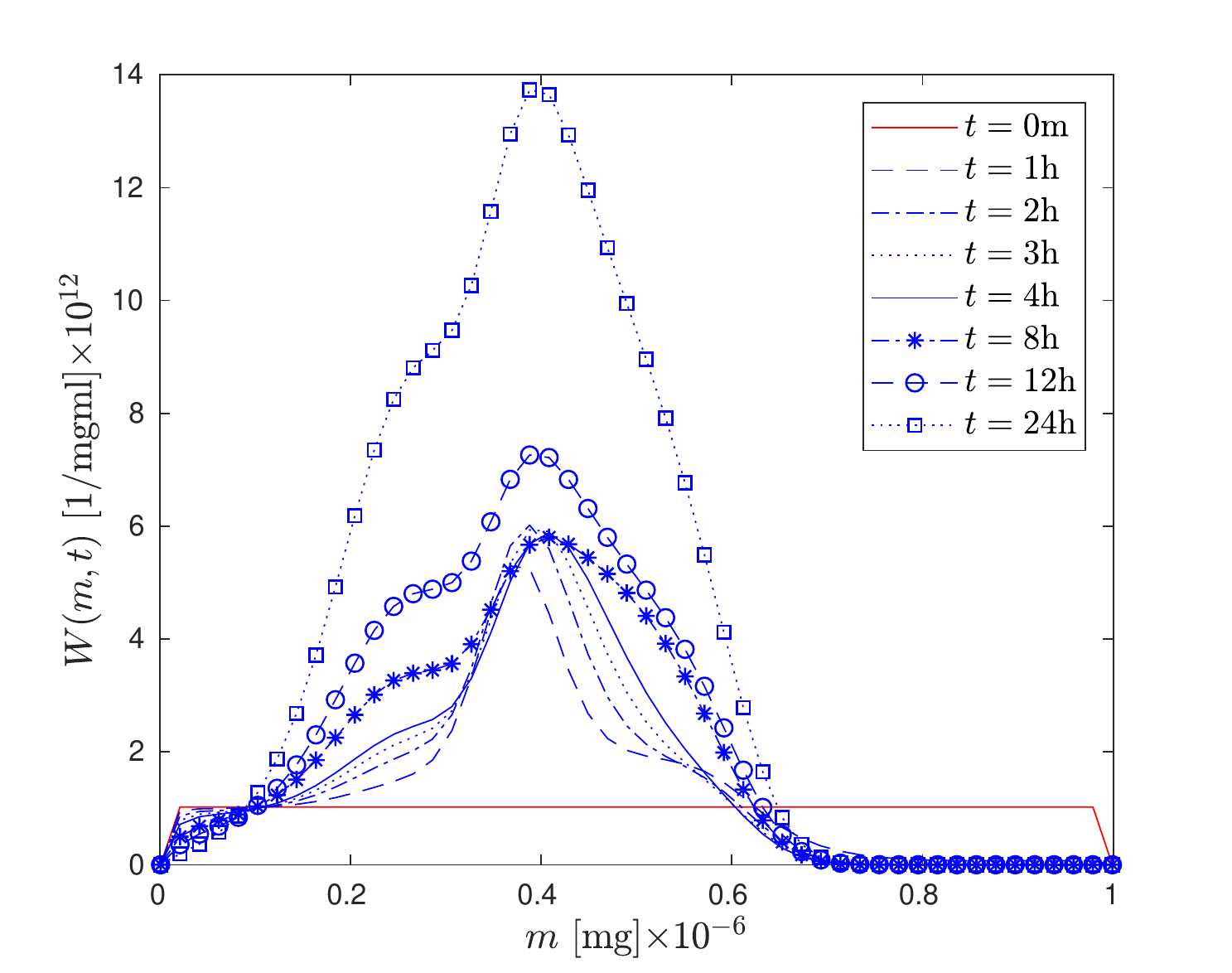}
			\caption{Mass discretization with 50 mass cells and time discretization with $h=\dfrac{1}{72}\approx0.0139$.}
		\end{subfigure}
		\begin{subfigure}[c]{0.49\textwidth}
			\vspace{-1em}
			\includegraphics[width=17em]{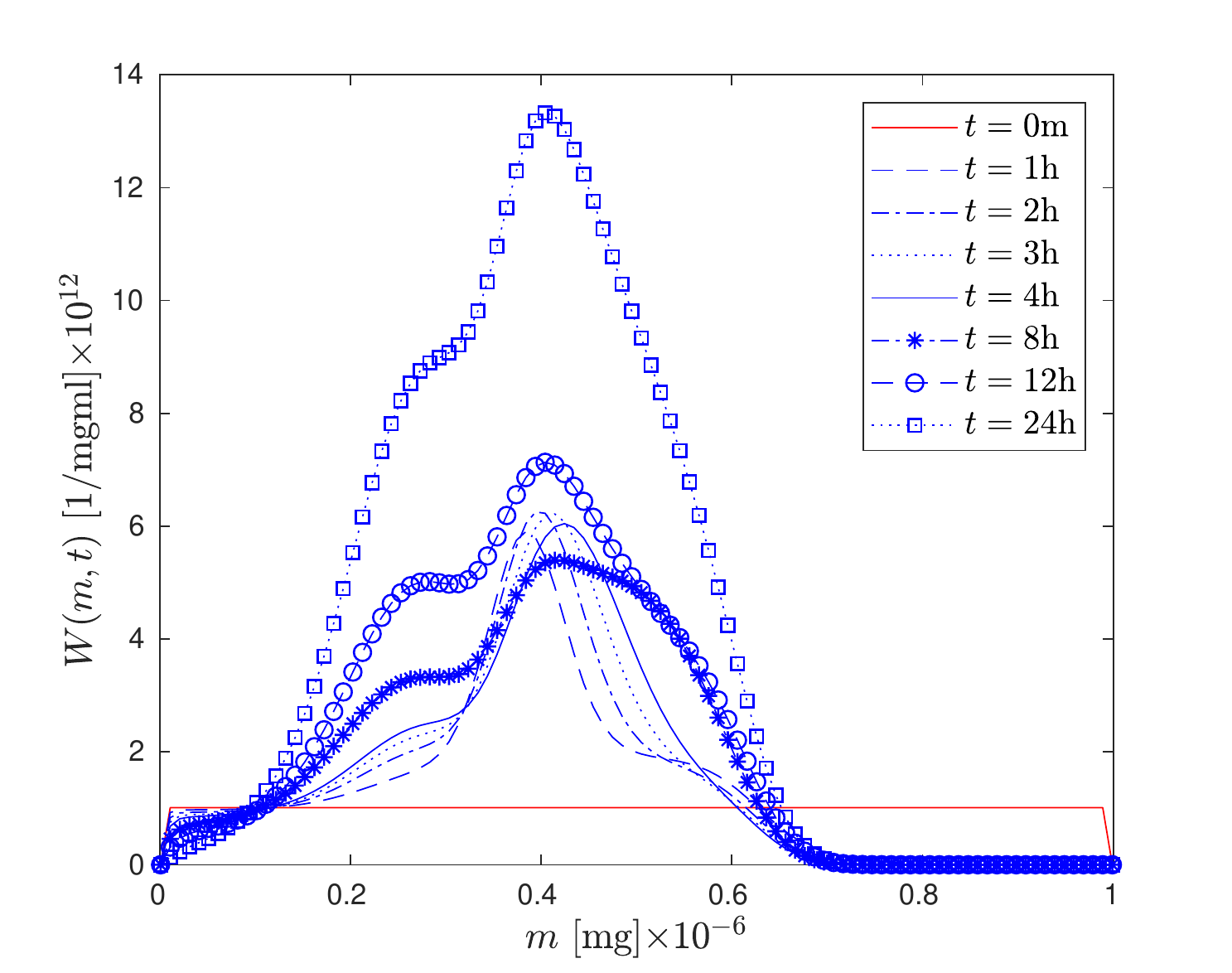}
				\caption{Mass discretization with 100 mass cells and time discretization with $h=\dfrac{1}{144}\approx0.0069$.}
			\end{subfigure}
			\begin{subfigure}[c]{0.49\textwidth}
					\vspace{-1em}
				\includegraphics[width=17em]{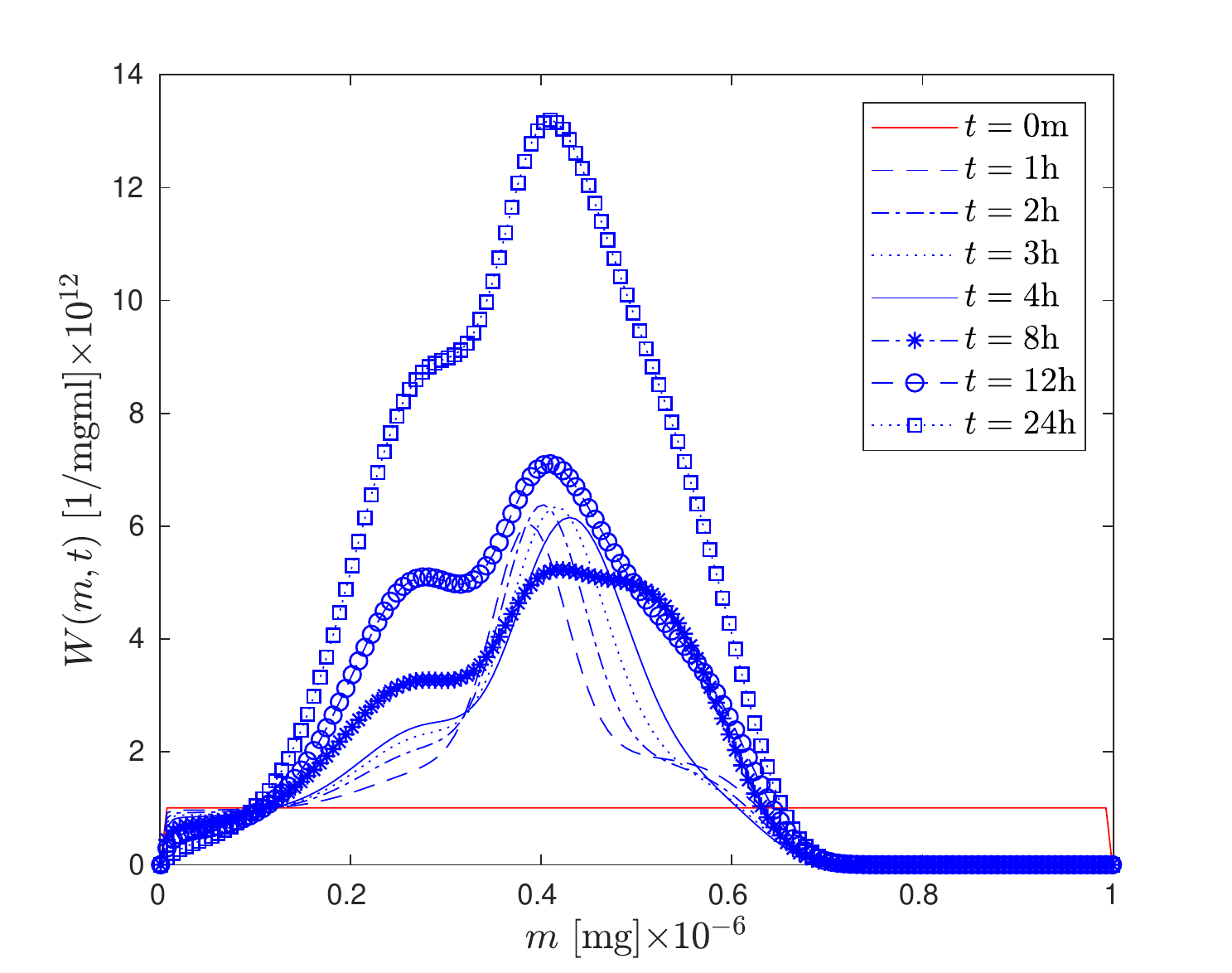}
					\caption{Mass discretization with 150 mass cells and time discretization with $h=\dfrac{1}{192}\approx0.0052$.}
				\end{subfigure}
				\centering
				\caption{Cell number density for different cell masses for first 24 hours  with constant initial distribution and implicit trapezoidal rule: Mass discretization with 30, 50, 100, 150 cells.}
				\label{Fig:CNDfirstCNconstant3050100150masscells}
			\end{figure}Figure~\ref{Fig:CNDfirstCNconstant3050100150masscells} shows cell number densities for the first 24 hours at different points in time. Thereby in  Figure~\ref{Fig:CNDfirstCNconstant3050100150masscells}, the discretization in the mass component is realized with 30, 50, 100 and 150 cells and the discretization in time with a time step of $h=\frac{1}{48}\approx0.0208$, $h=\frac{1}{72}\approx0.0139$, $h=\frac{1}{144}\approx0.0069$ and $h=\frac{1}{192}\approx0.0052$ respectively.
In Figure~\ref{Fig:CNDfirstCNconstant3050100150masscells}, for thirty mass cells a trend is already visible and more cells basically yield finer trajectories.\\
Figure~\ref{Fig:CND20CNconstant3050100150masscells} approves this observation for the cell number density at more time instances along the whole time horizon. The computation times here range from 33.38 to 180.45 CPUs for 30 to 150 cells.\\
	\begin{figure}[htbp]
				\centering
	\begin{subfigure}[c]{0.49\textwidth}
		\includegraphics[width=17em]{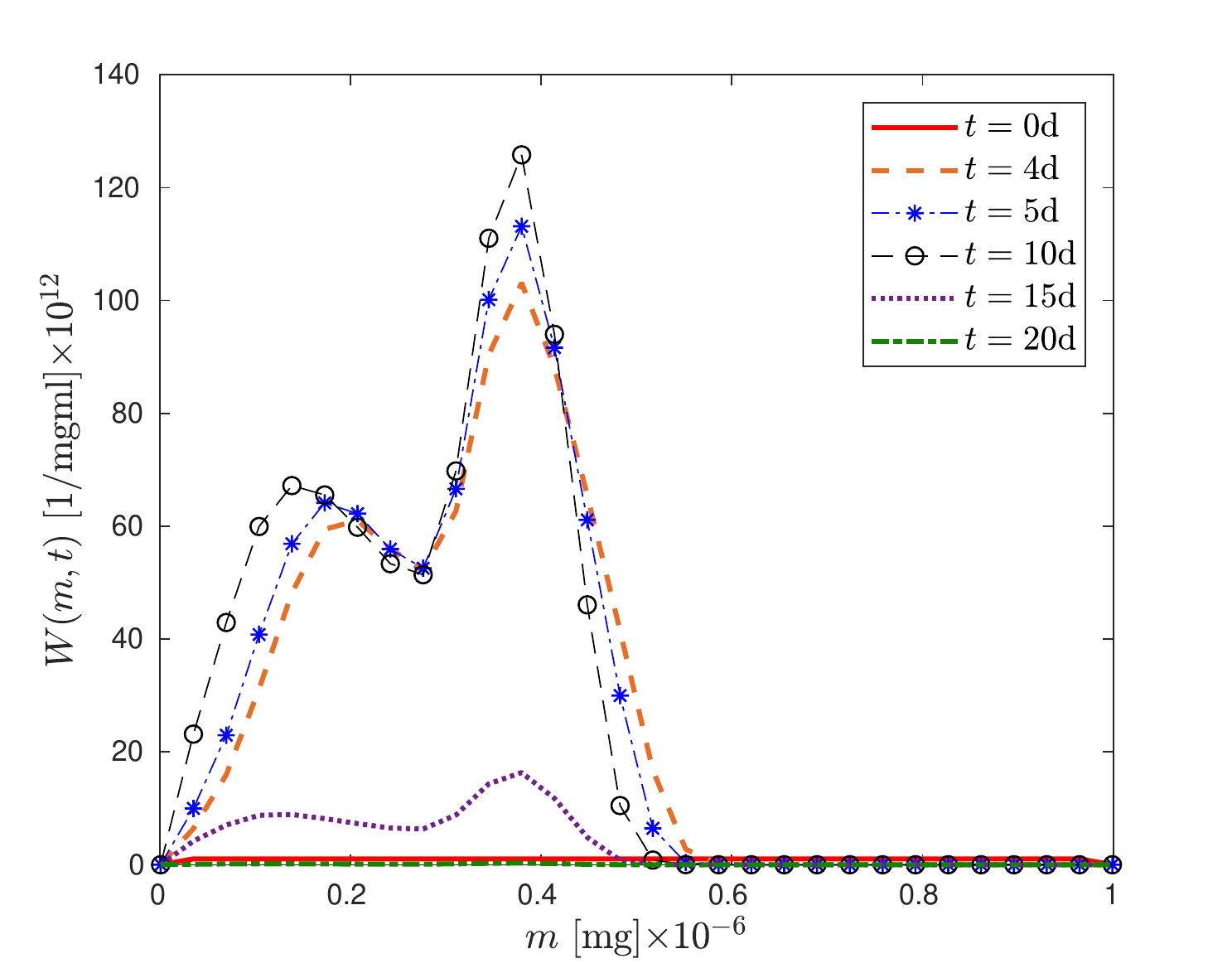}
		\caption{Mass discretization with 30 mass cells and time discretization with $h=\dfrac{1}{48}\approx0.0208$.}
	\end{subfigure}
	\begin{subfigure}[c]{0.49\textwidth}
		\includegraphics[width=17em]{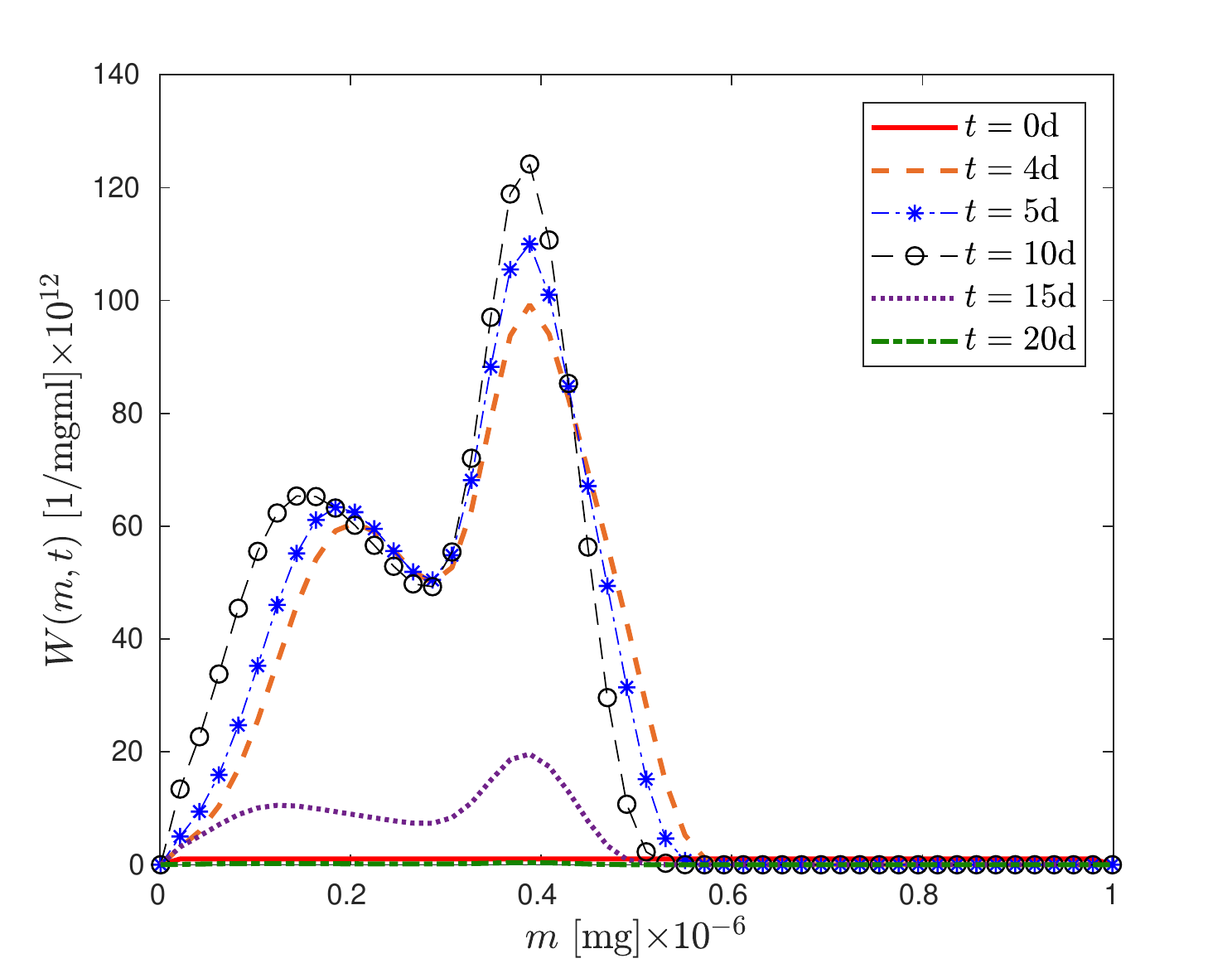}
		\caption{Mass discretization with 50 mass cells and time discretization with $h=\dfrac{1}{72}\approx0.0139$.}
	\end{subfigure}
	\begin{subfigure}[c]{0.49\textwidth}
		\vspace{-1em}
		\includegraphics[width=17em]{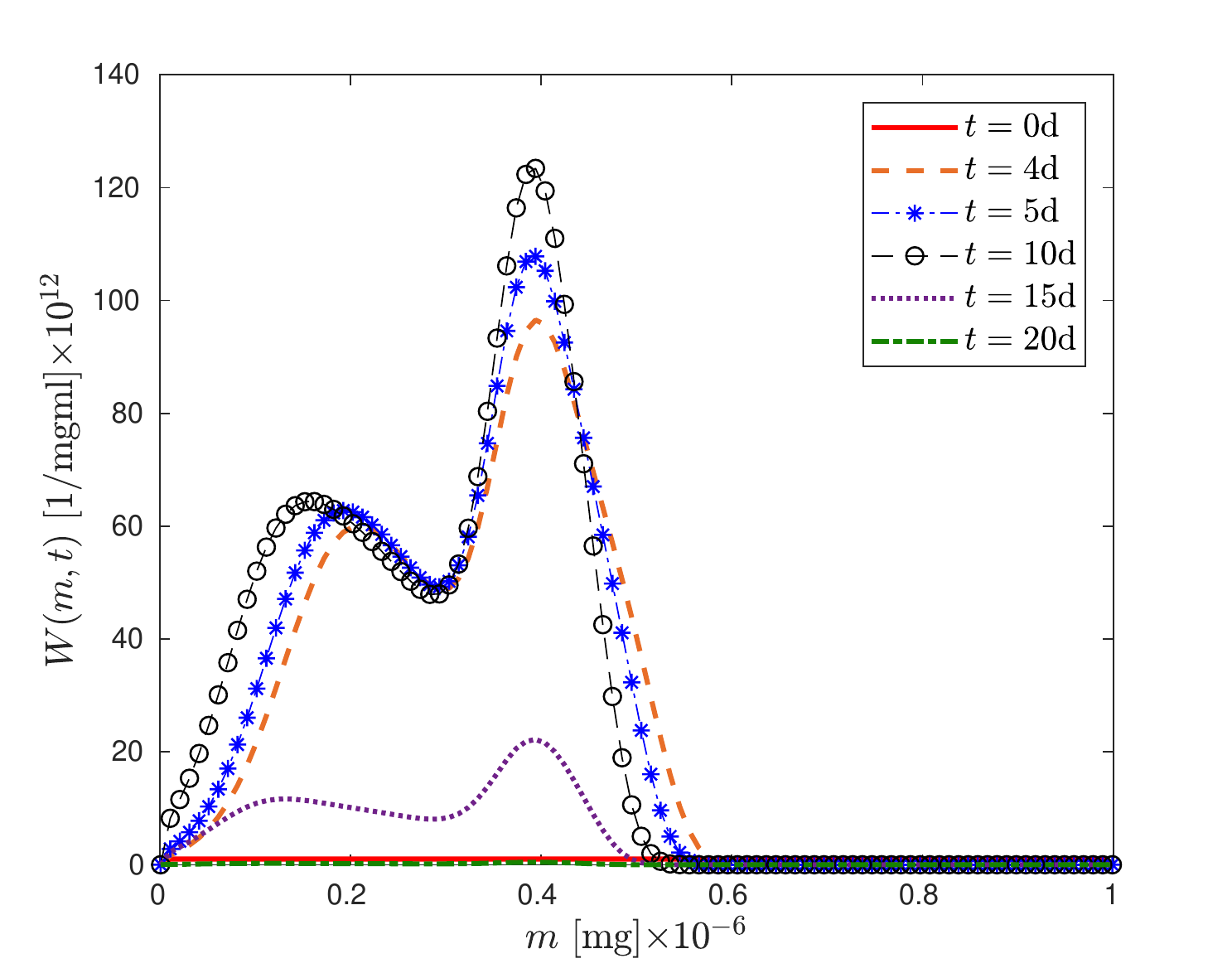}
		\caption{Mass discretization with 100 mass cells and time discretization with $h=\dfrac{1}{144}\approx0.0069$.}
	\end{subfigure}
	\begin{subfigure}[c]{0.49\textwidth}
		\vspace{-1em}
		\includegraphics[width=17em]{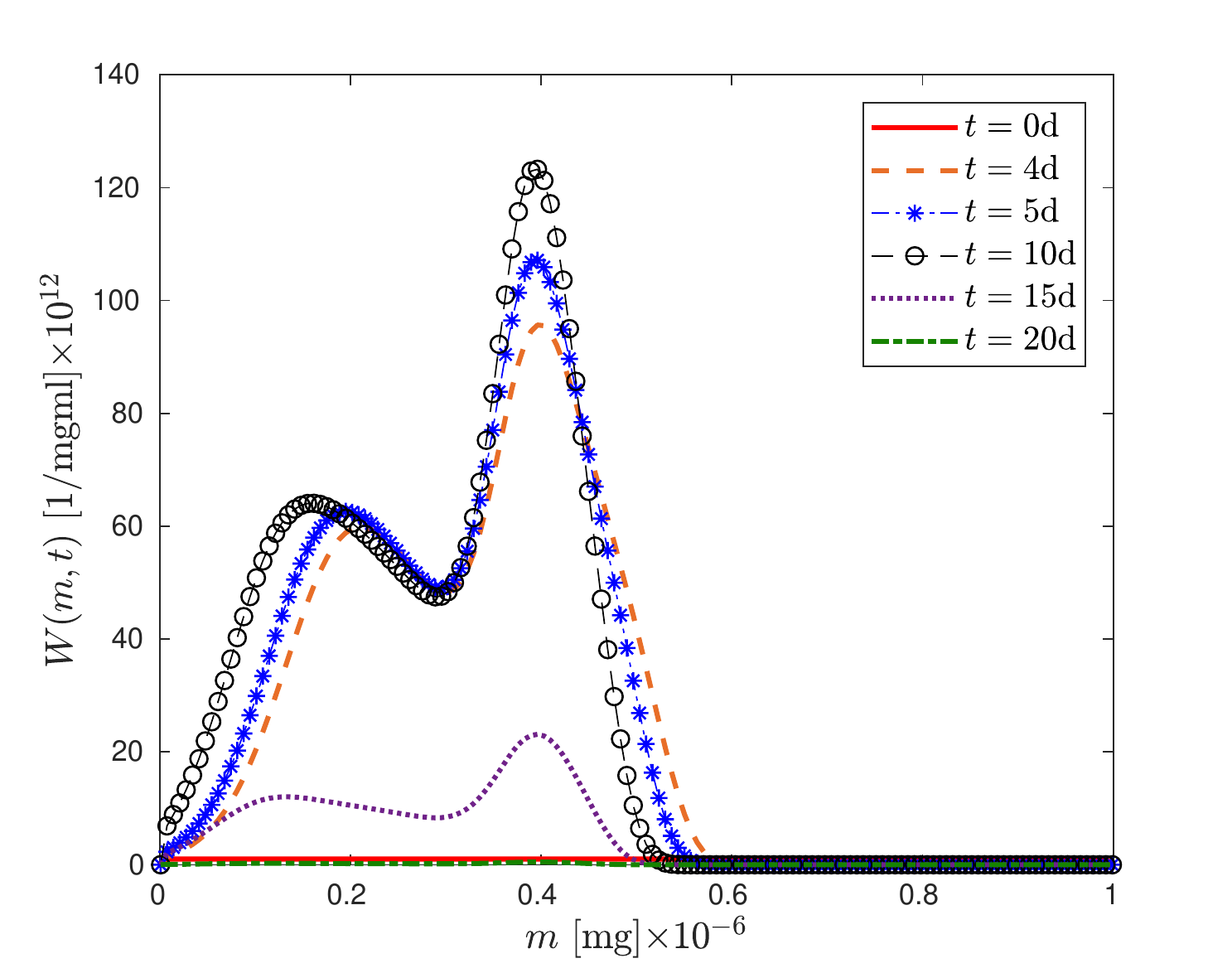}
		\caption{Mass discretization with 150 mass cells and time discretization with $h=\dfrac{1}{192}\approx0.0052$.}
	\end{subfigure}
	\centering
	\caption{Cell number density for different cell masses for twenty days with constant initial distribution and implicit trapezoidal rule: Mass discretization with 30, 50, 100, 150 cells.}
	\label{Fig:CND20CNconstant3050100150masscells}
\end{figure}From now on, $150$ mass cells for the mass discretization are used and as above the time step is always chosen with respect to the CFL condition and with respect to the convergence of the Newton's method.
	\begin{figure}[htbp]
				\centering
	\begin{subfigure}[c]{0.49\textwidth}
		\includegraphics[width=17em]{Imagesnew/ResultsPlotsCN/figure10constant150-eps-converted-to}
		\caption{Constant initial distribution.}
	\end{subfigure}
	\begin{subfigure}[c]{0.49\textwidth}
		\includegraphics[width=17em]{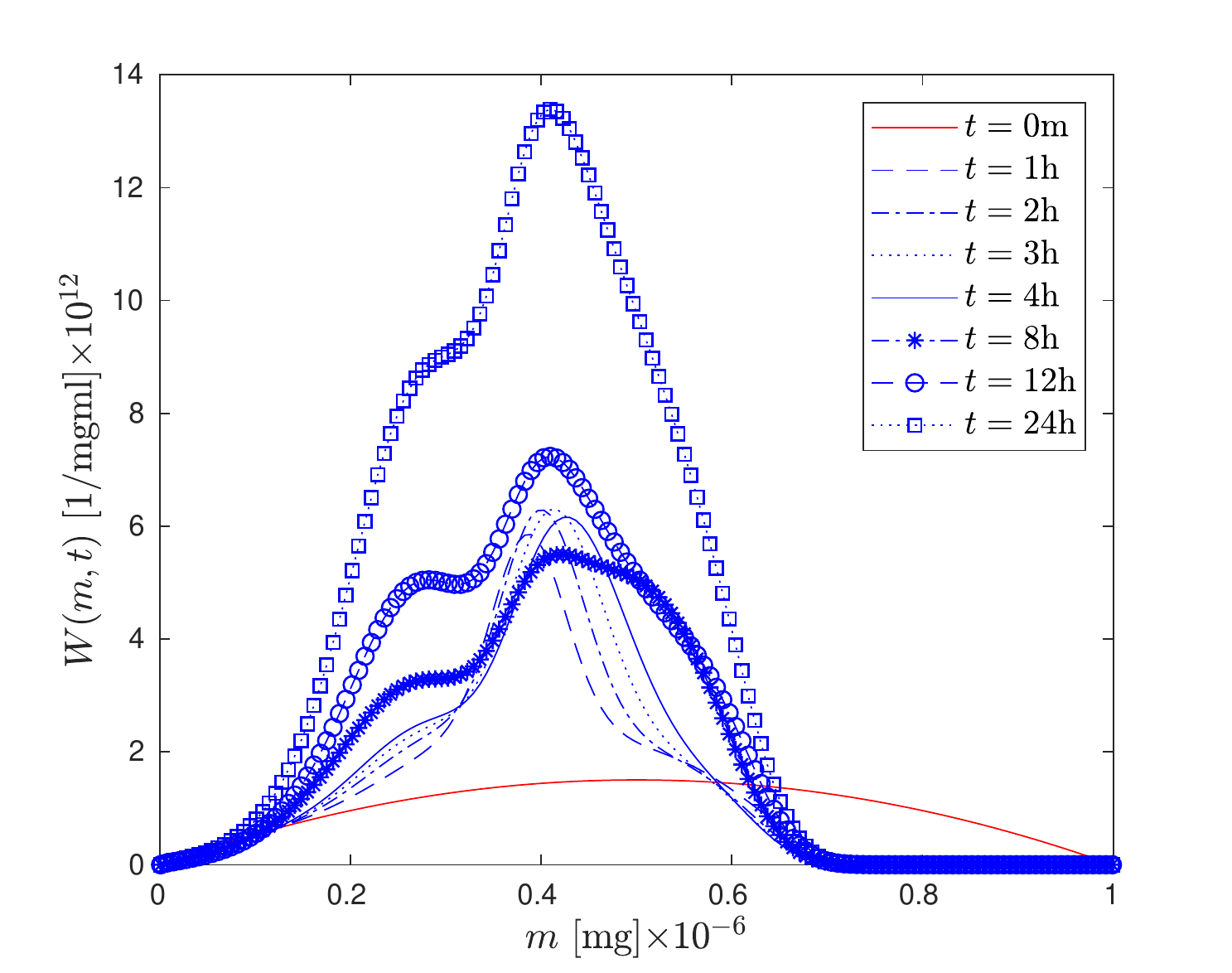}
		\caption{Beta initial distribution.}
	\end{subfigure}
	\begin{subfigure}[c]{0.49\textwidth}
		\vspace{-1em}
		\includegraphics[width=17em]{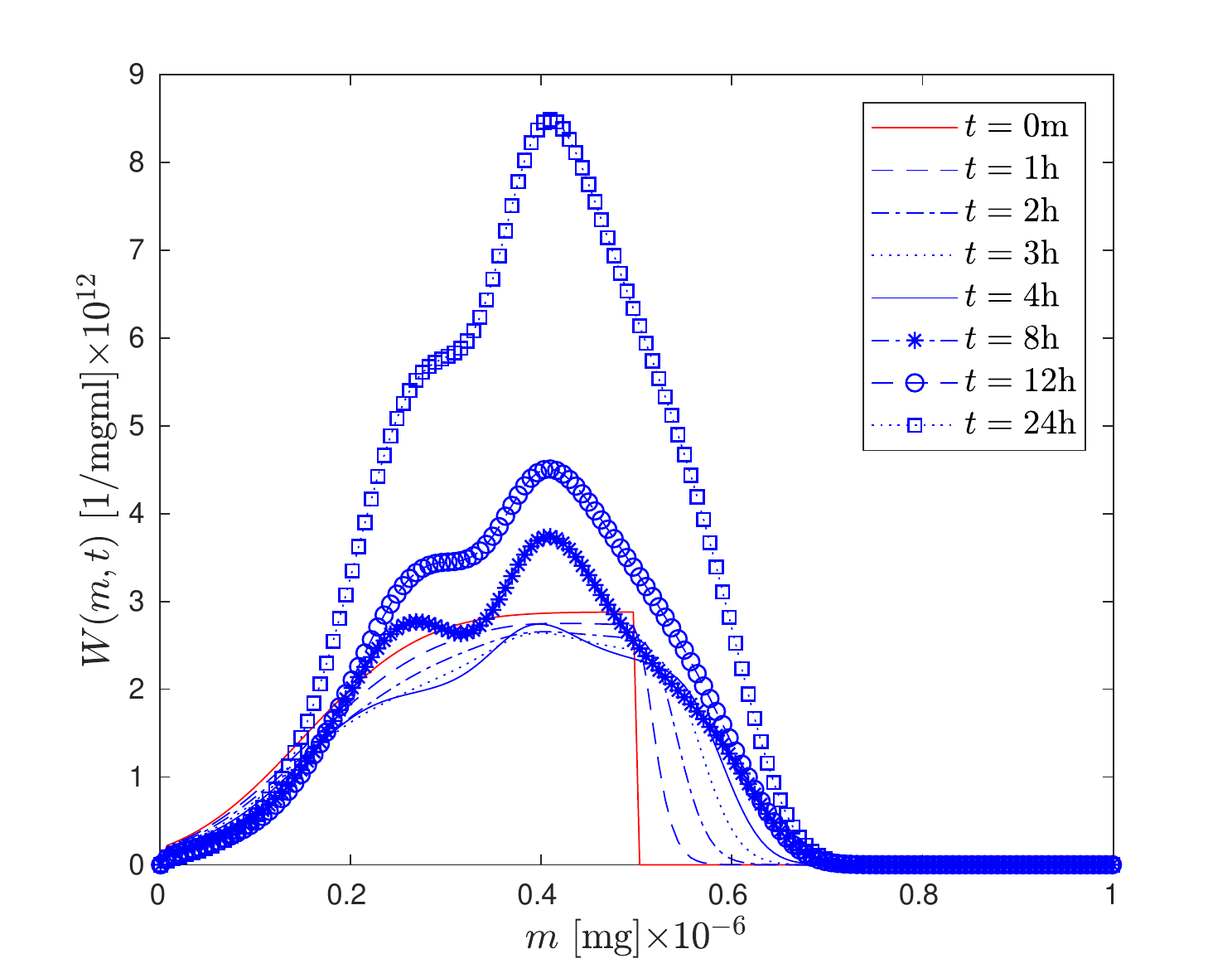}
		\caption{Small to medium cell initial distribution.}
	\end{subfigure}
	\begin{subfigure}[c]{0.49\textwidth}
		\vspace{-1em}
		\includegraphics[width=17em]{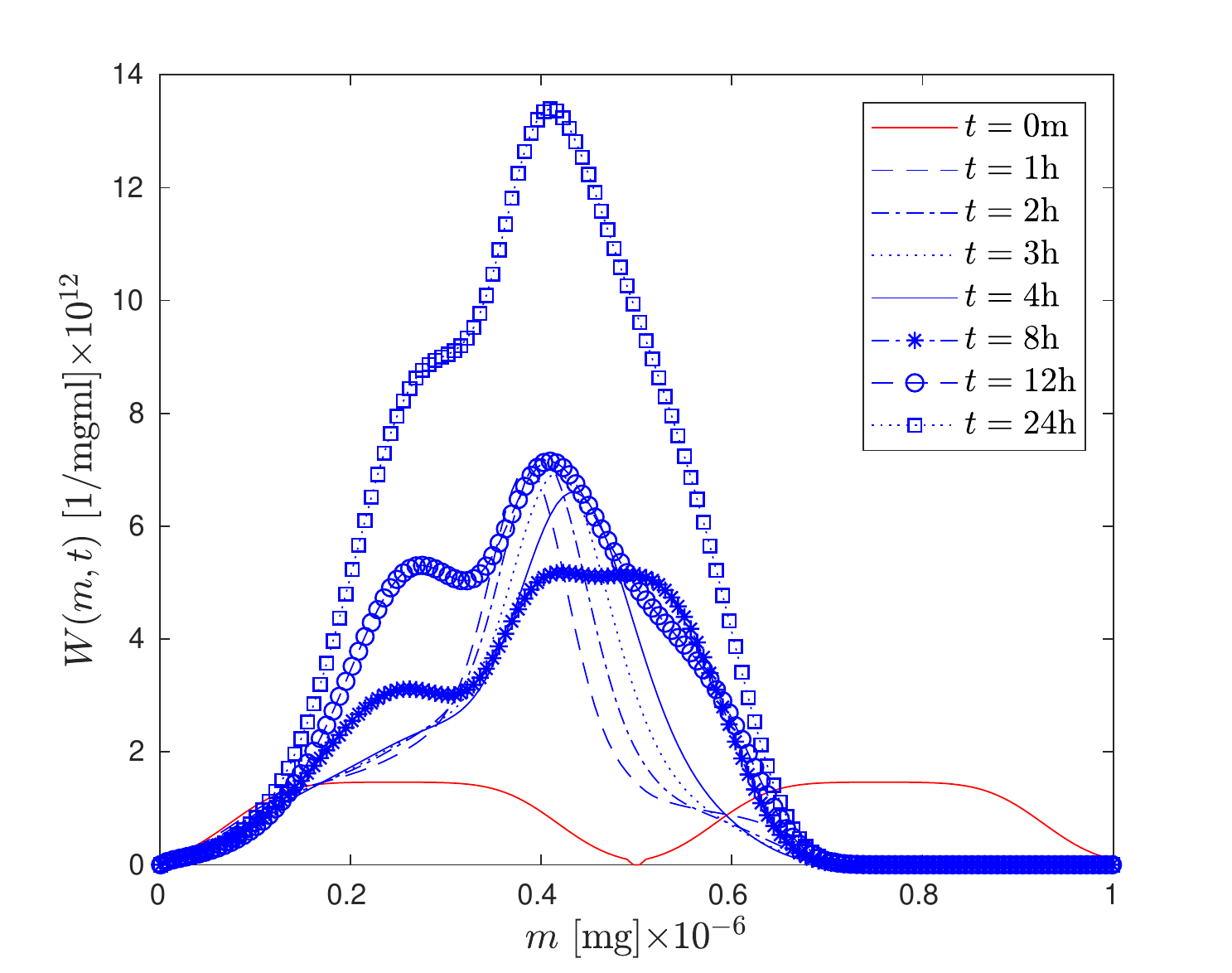}
		\caption{Two normal peak initial distribution.}
	\end{subfigure}
	\centering
	\caption{Cell number density for different cell masses for twenty days with different cell initial distributions and implicit trapezoidal rule: Mass discretization with 150 cells and time discretization with $h=\dfrac{1}{192}\approx0.0052$.}
	\label{Fig:CNDfirstCNdiffdist150masscells}
\end{figure}
	\begin{figure}[htbp]
				\centering
	\begin{subfigure}[c]{0.49\textwidth}
		\includegraphics[width=17em]{Imagesnew/ResultsPlotsCN/figure11constant150-eps-converted-to}
		\caption{Constant initial distribution.}
	\end{subfigure}
	\begin{subfigure}[c]{0.49\textwidth}
		\includegraphics[width=17em]{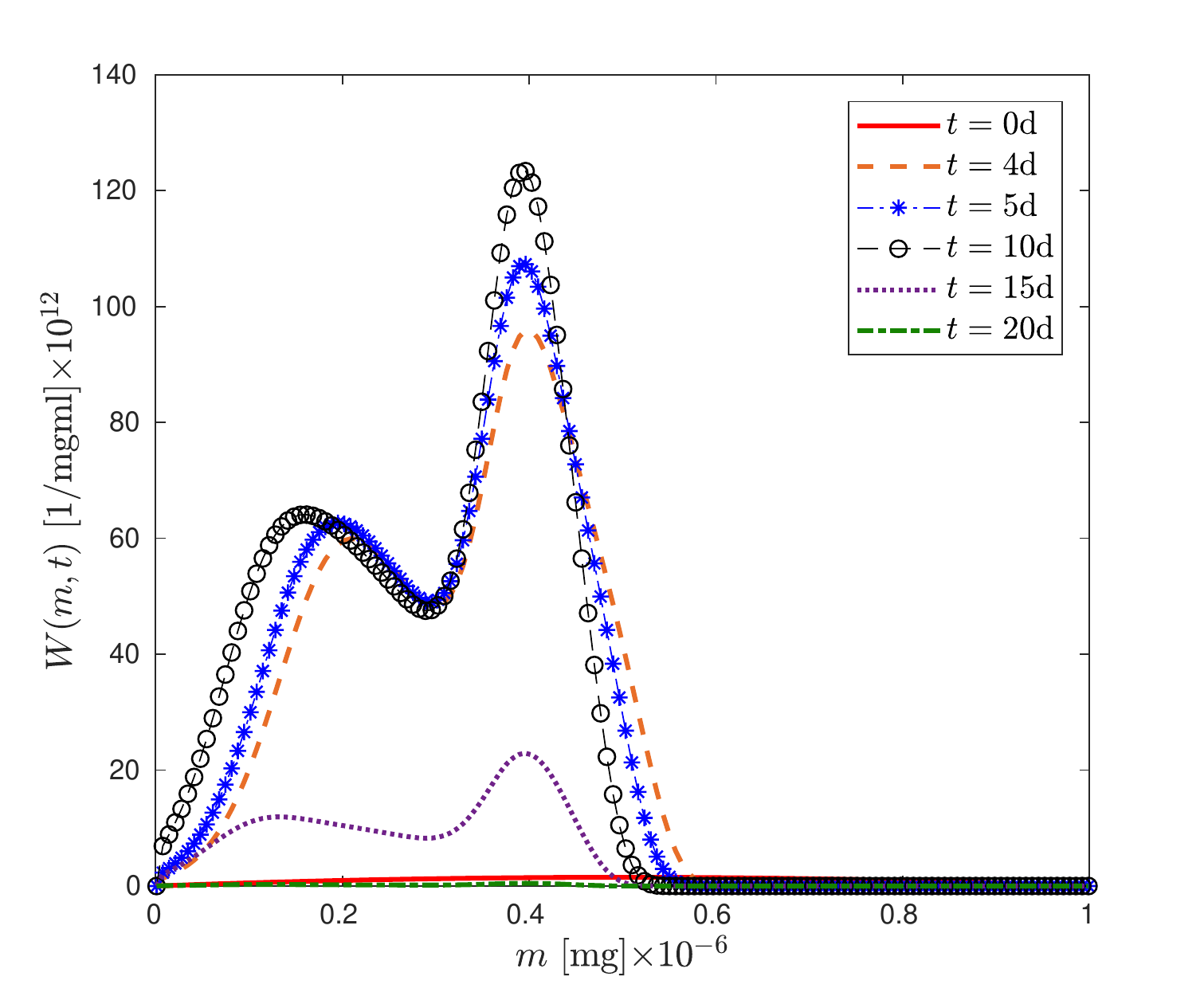}
		\caption{Beta initial distribution.}
	\end{subfigure}
	\begin{subfigure}[c]{0.49\textwidth}
		\vspace{-1em}
		\includegraphics[width=17em]{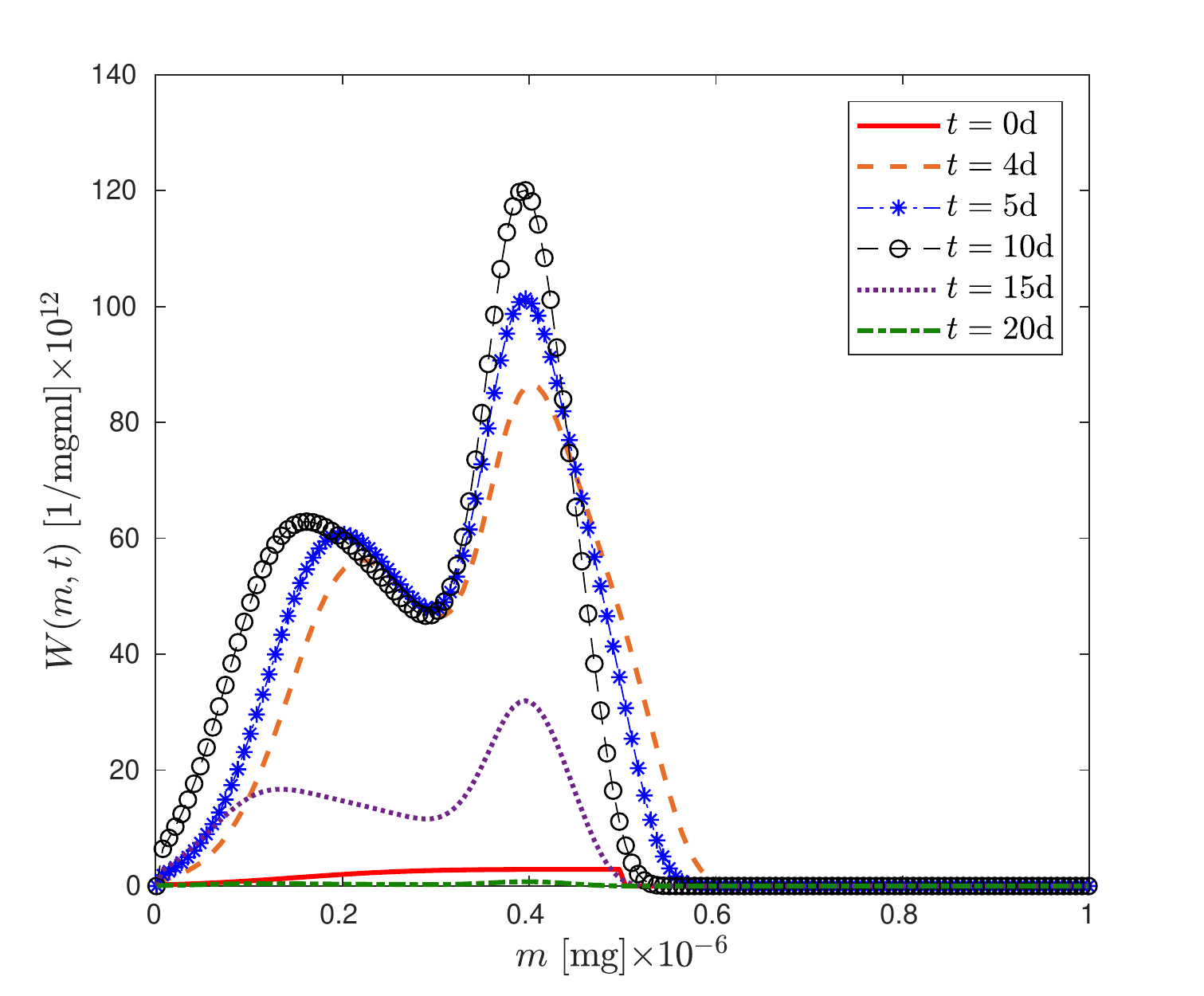}
		\caption{Small to medium cell initial distribution.}
	\end{subfigure}
	\begin{subfigure}[c]{0.49\textwidth}
		\vspace{-1em}
		\includegraphics[width=17em]{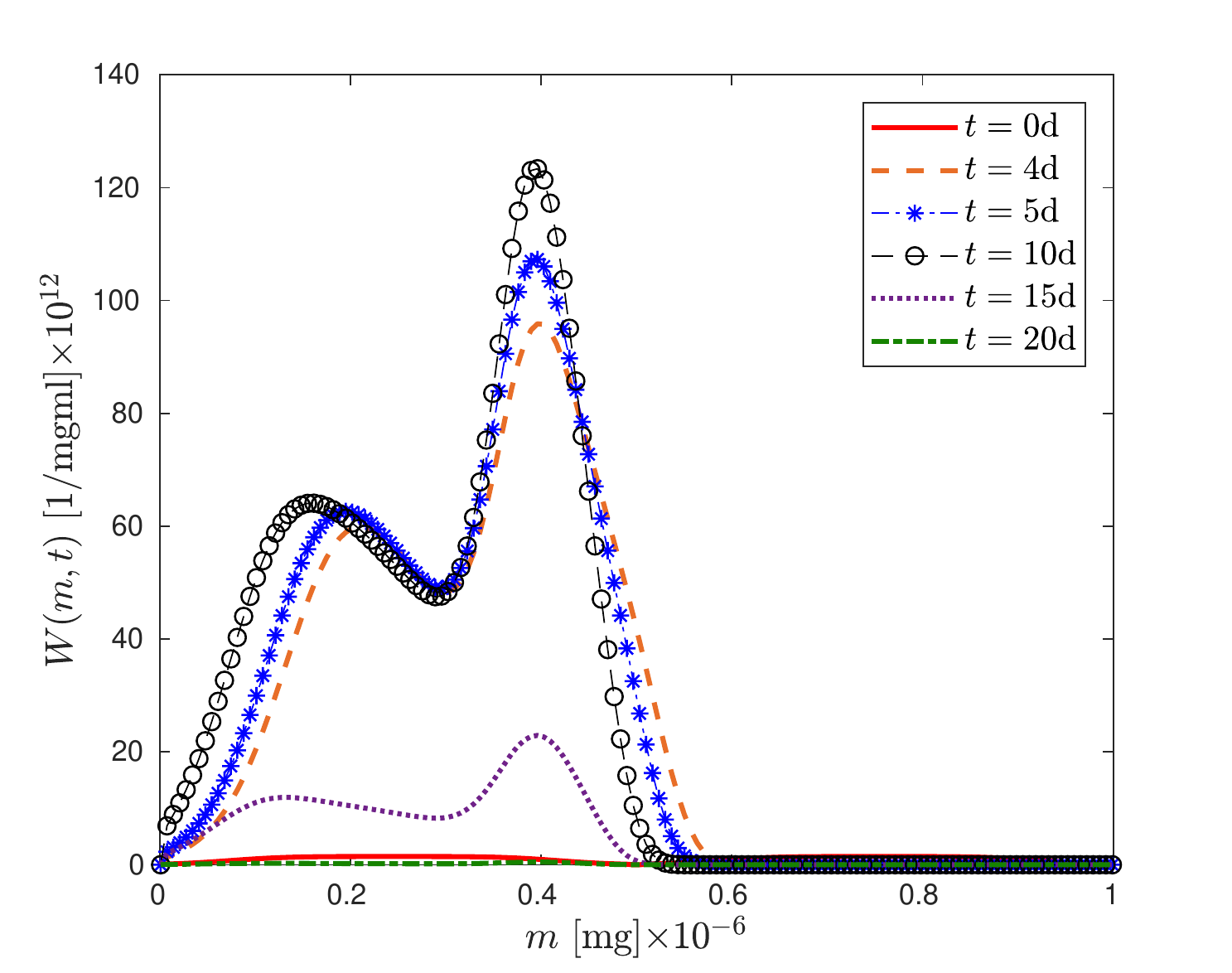}
		\caption{Two normal peak initial distribution.}
	\end{subfigure}
	\centering
	\caption{Cell number density for different cell masses for twenty days with different cell initial distributions and implicit trapezoidal rule: Mass discretization with 150 cells and time discretization with $h=\dfrac{1}{192}\approx0.0052$.}
	\label{Fig:CND20CNdiffdist150masscells}
\end{figure}In Figures~\ref{Fig:CNDfirstCNdiffdist150masscells} and \ref{Fig:CND20CNdiffdist150masscells}, the cell number density with respect to the different cell masses for some time instances in the first 24 hours and for the twenty days of fermentation is illustrated for different cell initial distributions respectively, i.e. a constant initial distribution, a beta initial distribution, a small to medium cell initial distribution and a two normal peak initial distribution. In Figure~\ref{Fig:CNDfirstCNdiffdist150masscells}, all the large cells are gone after two hours, i.e.\ cell division already took place and they were divided into daughter and mother cells again as explained in Section~\ref{sec:IDEmod}. According to \citet{morgan2007cell}, experiments have shown that a cell-cycle lasts $90$ to $120$ minutes or in other words yeast cells divide every $90$ to $120$ minutes.\\ The trajectories for the different initial distributions look very similar for Figures~(a)-(c) in Figure~\ref{Fig:CNDfirstCNdiffdist150masscells} but differ more remarkably for (c), in particular for the first 12 hours.\\ 
For the next days, as illustrated in Figure~\ref{Fig:CND20CNdiffdist150masscells}, the trajectories do not visibly differ much either apart from the one for the small to medium cell initial distribution. Overall, for all distributions two peaks appear which grow until day ten. One of these two peaks forms for cells of small mass and the other one for cells of medium mass until the point where division can take place. Later in time, more and more cells die due to the increasing ethanol concentration. The peak for the smaller cells becomes larger than the one for the medium cells. After twenty days there are mainly just small cells left.\\
\begin{figure}[htbp]
	\centering
	\includegraphics[scale=.6]{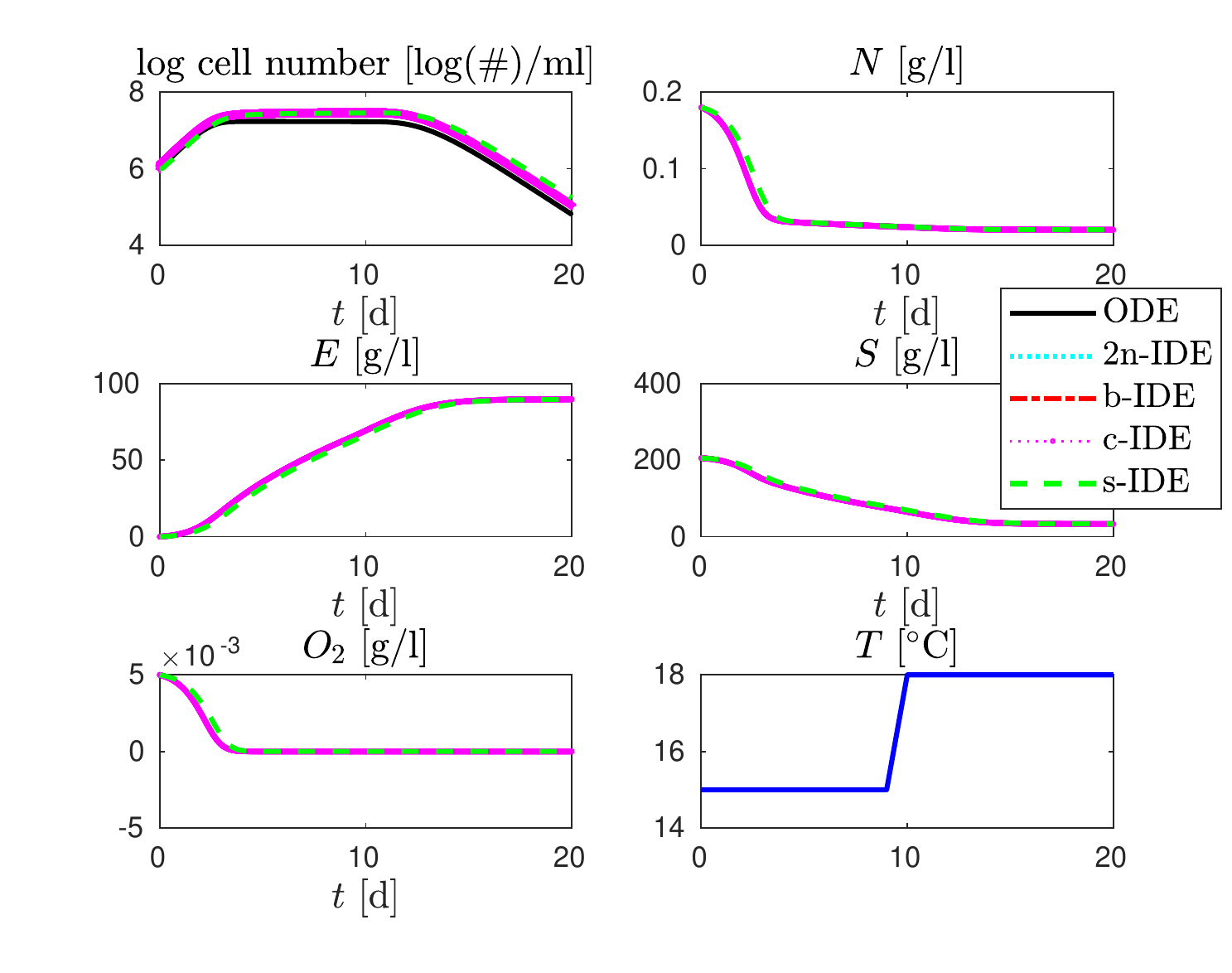}
	\centering
	\caption{Comparison of log cell number, substrate/product concentration trajectories, temperature profile for entire time horizon for IDE model (2n-IDE: 2 normal peak initial distribution, b-IDE: beta initial distribution, s-IDE: small to medium cell initial distribution) and ODE model (ODE).}
	\label{Fig:CNDallCNdiffdist150masscells}
\end{figure}Figure~\ref{Fig:CNDallCNdiffdist150masscells} illustrates the state trajectories, i.e.\ the logarithm of the cell number in $\log(\#)$/ml for the yeast and all other substrate concentration developments in g/l. For both methods, the yeast goes through all its growth phases apart from the lag phase which was not included in the model. The yeast growth phases can be found in \citet[Chapter~3]{mikrowein}. Sugar is consumed for yeast activity and its conversion into ethanol until there is only 18 g/l residual sugar left after twenty days and ethanol is accumulated up to a final ethanol concentration of 99 g/l which corresponds to approximately $12.5\%$. The oxygen is consumed within the first few days, from where nitrogen is consumed more slowly and approximately $0.019$ g/l still remains after twenty days of fermentation time.\\
For all of the distributions, the initial cell amount comprises 1 million cells per ml. In Figure~\ref{Fig:CNDallCNdiffdist150masscells} the four initial distributions (2n-IDE, b-IDE, c-IDE, s-IDE) already considered for Figure~\ref{Fig:CNDfirstCNdiffdist150masscells} and Figure~\ref{Fig:CND20CNdiffdist150masscells} are compared to the results coming from the model based on ODEs without modeling the yeast cell dynamics on a single-cell level coming from the simulation of a comparable ODE model version \citep{schenk17, schenk2018docthesis}. The main differences occur for yeast growth and yeast death but overall the differences are insignificant. The computation times here range from 116.29s to 180.37 CPUs for the different IDE initial distributions with s-IDE being the least computationally expensive and c-IDE being the most computationally expensive (b-IDE: 137.73 CPUs, 2n-IDE: 124.02 CPUs). In contrast, the solution of the ODE model only takes 0.39 CPUs. Thus, the IDE model computation is about 298 to 462 times as expensive as the ODE model computation.

			\section{Conclusions}\label{sect:concl}
			To determine the impact of the yeast cell mass initial distribution and the impact of the cell dynamics in general, a new model, based on a combination components previously introduced in the literature, was proposed.
			This model describes the reaction kinetics based on Micha\"elis-Menten kinetics and includes the oxygen and ethanol-dependent death. Furthermore it describes the single-cell behavior of the yeast with respect to growth, cell division and death. The resulting population balance model is based on strongly nonlinear weakly hyperbolic partial/ordinary integro-differential equations.
			For its solution a numerical solution scheme based on a finite volume approach was introduced and combined with the implicit trapezoidal rule. Existence and uniqueness of the solution of a simplified version investigating a semilinear population balance model was shown based on semigroup theory. The more complex semilinear case with linear velocity in $m$ and the quasilinear case are currently under investigation.\\
			The numerical scheme was applied to the proposed model and the corresponding numerical results were presented with respect to a comparison of different cell mass initial distributions and a comparison to a reduced ODE model.
			The results for a constant, beta, small to medium and a two normal peak cell initial distribution were compared. The results reveal that the impact of the initial distribution is smaller than expected. The cell number densities differ significantly for the first 24 hours but in the long run differ hardly.
			The simulation results for the IDE model with the different initial cell distributions were compared to results without modeling the yeast cell dynamics on a single-cell level coming from the simulation of a comparable ODE model version. The main differences occur for the yeast growth and yeast death. Thus, the impact of modeling the cell dynamics is much less than expected and almost negligible. So, it is questionable whether it brings much more value to use this as the descriptive process model with respect to process optimization. Moreover, this would also require more data related to the dynamics in order to estimate the dynamic parameters. 
			To collect this kind of data is very expensive or not even possible in the scope as it would be useful. Another bottleneck of using this model for process optimization is that its solution is computationally expensive.\\
			Nevertheless, the investigated model is very interesting from a mathematical point of view. The outcomes of this work can be useful for other models based on integro-differential equations from other fields of application in finance, engineering sciences including physical and other biological processes.\\
			Moreover, the outcomes for the particular wine fermentation model studied here can be very useful for winemakers in understanding the evolution of the yeast cell population with respect to cell mass and can support process engineers with useful mathematical investigations and making a process model choice. This is essential for process optimization, i.e. for increasing profit and process efficiency.
			\appendix
			\section*{Appendix A. The Jacobian of the Right Hand Side $f$ in Section~\ref{sect:resultsdiscuss}}\label{app:Jacobianf}
We start with the first component of $f$, namely $f_{w_i}$, and form the derivatives with respect to all other components, in detail
\begin{align}
\dfrac{\partial f_{w_i}}{\partial w_{i}}&= \dfrac{1}{\Delta m}\left(-r_{\epsilon}(m_{i+1},N,S,O)-2\sum_{j=1}^{N_W-1} K_{ij}-\int_{m_i}^{m_{i+1}}\Gamma(m)\; dm-\Phi(E)-k_d\right),\notag\\[1em]
\dfrac{\partial f_{w_i}}{\partial w_{i-1}}&=\dfrac{1}{\Delta m}(r_{\epsilon}(m_i,N,S,O)),\notag\\[1em]
\dfrac{\partial f_{w_i}}{\partial N}&= -\dfrac{1}{\Delta m}\left(\left(-\dfrac{r_{\epsilon}(m_{i+1},N,S,O))}{K_N+N}+\dfrac{S\mu_{max}(T)m_{i+1}}{(K_N+N)(K_{S_1}+S)}\left(\dfrac{O}{K_O+O}+\epsilon\right)\right)w_i\right.\notag\\&\hspace{0.5cm} \left.+\left(\dfrac{r_{\epsilon}(m_i,N,S,O)}{K_N+N}-\dfrac{S\mu_{max}(T)m_{i}}{(K_N+N)(K_{S_1}+S)}\left(\dfrac{O}{K_O+O}+\epsilon\right)\right)w_{i-1}\right),\notag\\[1em]
\dfrac{\partial f_{w_i}}{\partial E}&=  k_{d_1}k_{d_2}(E-tol)^2/(\pi(1+{k_{d_1}}^2(E-tol)^2))\notag\\&\hspace{0.5cm} +2k_{d_2}(E-tol)(0.5+\atan(k_{d_1}(E-tol))/\pi) w_i,\notag\\[1em]
\dfrac{\partial f_{w_i}}{\partial S}&= -\dfrac{1}{\Delta m}\left(\left(-\dfrac{r_{\epsilon}(m_{i+1}, N, S, O)}{(K_{S_1}+S)}+\dfrac{N\mu_{max}(T)m_{i+1}}{((K_N+N)(K_{S_1}+S))}\left(\dfrac{O}{K_O+O}+\epsilon\right)\right)w_i\right.\notag\\&\hspace{0.5cm} \left.+\left(\dfrac{r_{\epsilon}(m_i, N, S, O)}{(K_{S_1}+S)}-\dfrac{N\mu_{max}(T)m_i}{((K_N+N)(K_{S_1}+S))}\left(\dfrac{O}{K_O+O}+\epsilon\right)\right)w_{i-1}\right)\quad\text{and}\notag\\[1em]
\dfrac{\partial f_{w_i}}{\partial O}&= -\dfrac{1}{\Delta m}\left(\left(-\dfrac{r(m_{i+1}, N, S, O)}{(K_O+O)}+\dfrac{SN\mu_{max}(T)m_{i+1}}{((K_N+N)(K_O+O)(K_{S_1}+S))}\right)w_i\right.\notag\\&\hspace{0.5cm} \left.+\left(\dfrac{r(m_i, N, S, O)}{(K_{O}+O)}-\dfrac{SN\mu_{max}(T)m_i}{((K_N+N)(K_O+O)(K_{S_1}+S))}\right)w_{i-1}\right).\notag
\end{align}
Let us continue with the derivatives of $f_N$ with respect to all states apart from $E$, i.e.\
\begin{align}
\dfrac{\partial f_N}{\partial w_i}&=-k_1\tilde{r}_{\epsilon}(N,S,O)\frac{(m_i+m_{i+1})}{2}\Delta m,\notag\\[1em]
\dfrac{\partial f_N}{\partial N}&= -k_1\sum_{i=1}^{N_W-1}\Delta m \frac{(m_i+m_{i+1})}{2}w_i\left(-\dfrac{\tilde{r}_{\epsilon}(N,S,O)}{K_N+N}\right.\notag\\&\hspace{0.5cm} \left.+\dfrac{\mu_{max}(T)S}{(K_N+N)(K_{S_1}+S)}\left(\dfrac{O}{K_O+O}+\epsilon\right)\right),\notag\\[1em]
\dfrac{\partial f_N}{\partial S}&= -k_1\sum_{i=1}^{N_W-1}\Delta m \frac{(m_i+m_{i+1})}{2}w_i\left(-\dfrac{\tilde{r}_{\epsilon}(N,S,O)}{K_{S_1}+S}\right.\notag\\&\hspace{0.5cm} \left.+\dfrac{\mu_{max}(T)N}{(K_N+N)(K_{S_1}+S)}\left(\dfrac{O}{K_O+O}+\epsilon\right)\right)\quad\text{and}\notag %
\end{align}
\begin{align}
\dfrac{\partial f_N}{\partial O}&= -k_1\sum_{i=1}^{N_W-1}\Delta m \frac{(m_i+m_{i+1})}{2}w_i\left(-\dfrac{\tilde{r}(N,S,O)}{K_{O}+O}\right.\notag\\ &\hspace{0.5cm}\left.+\dfrac{\mu_{max}(T)NS}{(K_N+N)(K_O+O)(K_{S_1}+S)}\right).\notag
\end{align}
Moreover, the derivatives of $f_E$ with respect to $w_i$, $E$ and $S$ are represented by\begin{align}
\dfrac{\partial f_E}{\partial w_i}&= \tilde{q}(S,E)\frac{(m_i+m_{i+1})}{2}\Delta m,\notag\\[1em]
\dfrac{\partial f_E}{\partial E}&= \sum_{i=1}^{N_W-1}\Delta m \frac{(m_i+m_{i+1})}{2}w_i\left(\dfrac{(-\tilde{q}(S,E))}{E+K_E(T)}\right)\quad\text{and}\notag\\[1em]
\dfrac{\partial f_E}{\partial S}&= \sum_{i=1}^{N_W-1}\Delta m \frac{(m_i+m_{i+1})}{2}w_i\left(\dfrac{(-\tilde{q}(S,E))}{S+K_{S_2}}+\dfrac{\beta_{max}(T)K_E(T)}{(E+K_E(T))(K_{S_2}+S)}\right).\notag
\end{align}Furthermore, for the derivatives of $f_S$ with respect to all states we obtain
\begin{align}
\dfrac{\partial f_S}{\partial w_i}&= \left(-k_2\tilde{q}(S,E)-k_3\tilde{r}_{\epsilon}(N,S,O)\right)0.5(m_i+m_{i+1})\Delta m,\notag\\[1em]
\dfrac{\partial f_S}{\partial N}&= -k_3\sum_{i=1}^{N_W-1}\Delta m \frac{(m_i+m_{i+1})}{2}w_i\left(-\dfrac{\tilde{r}_{\epsilon}(N,S,O)}{K_N+N}\right.\notag\\&\hspace{0.5cm}\left.+\dfrac{\mu_{max}(T)S}{(K_N+N)(K_{S_1}+S)}\left(\dfrac{O}{K_O+O}+\epsilon\right)\right),\notag\\[1em]
\dfrac{\partial f_S}{\partial E}&= -k_2\sum_{i=1}^{N_W-1}\Delta m \frac{(m_i+m_{i+1})}{2}w_i\left(-\dfrac{\tilde{q}(S,E)}{E+K_E(T)}\right),\notag\\[1em]
\dfrac{\partial f_S}{\partial S}&= -k_3\sum_{i=1}^{N_W-1}\Delta m \frac{(m_i+m_{i+1})}{2}w_i\left(-\dfrac{\tilde{r}_{\epsilon}(N,S,O)}{K_{S_1}+S}\right.\notag\\&\hspace{0.5cm}\left.+\dfrac{\mu_{max}(T)N}{(K_N+N)(K_{S_1}+S)}\left(\dfrac{O}{K_O+O}+\epsilon\right)\right)\notag\\&\hspace{0.5cm} -k_2\sum_{i=1}^{N_W-1}\Delta m \frac{(m_i+m_{i+1})}{2}w_i\left(-\dfrac{\tilde{q}(S,E)}{S+K_{S_2}}+\dfrac{\beta_{max}(T)K_E(T)}{(E+K_E(T))(K_{S_2}+S)}\right)\;\;\text{and}\notag\\[1em]
\dfrac{\partial f_S}{\partial O}&= -k_3\sum_{i=1}^{N_W-1}\Delta m \frac{(m_i+m_{i+1})}{2}w_i\left(-\dfrac{\tilde{r}(N,S,O)}{K_O+O}\right.\notag\\&\left.+\dfrac{\mu_{max}(T)SN}{(K_N+N)(K_{S_1}+S)(K_O+O)}\right).\notag
\end{align}
Finally, the derivatives of $f_O$ with respect to $w_i$, $N$, $S$ and $O$ look like the following
\begin{align}
\dfrac{\partial f_{O}}{\partial w_i}&= -k_4\tilde{r}(N,S,O)\Delta m \frac{(m_i+m_{i+1})}{2},\notag\\[1em]
\dfrac{\partial f_{O}}{\partial N}&= -k_4\sum_{i=1}^{N_W-1}\Delta m \frac{(m_i+m_{i+1})}{2}w_i\left(-\dfrac{\tilde{r}(N,S,O)}{K_N+N}\right.\notag\\&\left.+\dfrac{\mu_{max}(T)SO}{(K_N+N)(K_O+O)(K_{S_1}+S)}\right),\notag\\[1em]
\dfrac{\partial f_{O}}{\partial S}&= -k_4\sum_{i=1}^{N_W-1}\Delta m \frac{(m_i+m_{i+1})}{2}w_i\left(-\dfrac{\tilde{r}(N,S,O)}{K_{S_1}+S}\right.\notag\\&\hspace{0.5cm}\left.+\dfrac{\mu_{max}(T)NO}{(K_N+N)(K_O+O)(K_{S_1}+S)}\right)\quad\text{and}\notag\\[1em]
\dfrac{\partial f_{O}}{\partial O}&= -k_4\sum_{i=1}^{N_W-1}\Delta m \frac{(m_i+m_{i+1})}{2}w_i\left(-\dfrac{\tilde{r}(N,S,O)}{K_{O}+O}+\right.\notag\\&\left.\dfrac{\mu_{max}(T)NS}{(K_N+N)(K_O+O)(K_{S_1}+S)}\right).\notag
\end{align}
All other derivatives apart from the derivatives of the boundary conditions are equal to zero.\\ The derivatives of the boundary conditions for the yeast population are of the following form
\begin{align}
\dfrac{\partial f_{w_0}}{\partial w_0}&=\dfrac{1}{\Delta m}r_{\epsilon}(m_{min},N,S,O).\notag
\end{align}
\section*{Acknowledgements}
\raggedright
This research work was funded by the German Federal Ministry of Education and Research (BMBF, Bundesministerium f{\"u}r Bildung und Forschung) within the collaborative project R\OE NOBIO (Robust energy-optimization of fermentation processes for the production of biogas and wine) with contract number 05M2013UTA. Moreover, it has been partly supported by the German Research Foundation (DFG) within the research training group 2126 Algorithmic Optimization. The authors would also like to gratefully acknowledge the funding by the Ministerio de Economía y Competitividad (MINECO) of the Spanish Government through BCAM Severo Ochoa accreditation SEV-2017-0718, and the funding by the Basque Government under the BERC 2018e2021 Program and the grant "Artificial Intelligence in BCAM" number EXP. 2019/00432 for postdoctoral fellowship (to C.S.). Moreover, the authors thank Leonhard Frerick for his support with his analytical expertise. Additionally, the authors are grateful to Simone Rusconi (Basque Center for Applied Mathematics) and Marta Lewicka (University of Pittsburgh) for some useful comments which improved this work. Furthermore, many thanks goes to our joint and associated partners within the project R\OE NOBIO.
\bibliography{literaturenumpaper}
\end{document}